\documentclass{article}[12pt]
\usepackage{amsmath,amssymb,amsthm,txfonts}
 \usepackage{color}

\title{Buchberger-Weispfenning Theory\\ for Effective Associative Rings}

\author{
{\small\bf Michela Ceria}\\
{\small Dipartimento di Ingegneria e scienza dell'informazione}\\
{\small Universit\`{a} di Trento}\\
{\small{\tt michela.ceria@unitn.it}}
\and
\and
{\small\bf Teo Mora}\\
{\small DIMA}\\
{\small Universit\`a di Genova}\\
{\small{\tt theomora@disi.unige.it}}
}

\def\Forall{\mbox{\ for each }}
\def\Where{\mbox{\ where }}

\def\And{\mbox{\ and }}

\def\Or{\mbox{\ or }}
\def\Bcc#1{\overline{\bf #1}}
\def\Bbb#1{{\mathbb #1}}
\def\Cal#1{{\cal #1}}
\def\Frak #1{{\mathfrak #1}}

\def\GM{{\mathfrak G}{\mathfrak M}}
\def\then{\;\Longrightarrow\;}

\def\Rep{\mathop{\bf Rep}\nolimits}

\def\Span{\mathop{\rm Span}\nolimits}
\def\supp{\mathop{\rm supp}\nolimits}
\def\lcm{\mathop{\rm lcm}\nolimits}
\def\lc{\mathop{\rm lc}\nolimits}
\def\lift{\mathop{\rm lift}\nolimits}

\def\NF{\mathop{\rm NF}\nolimits}
\def\tail{\mathop{\rm tail}\nolimits}
\def\op{\mathop{\rm op}\nolimits}

\def\GM{{\mathfrak G}{\mathfrak M}}
\begin{document}
\maketitle
\renewcommand\labelitemi{\bf --}
\def\qed{\ifmmode\squareforqed\else{\unskip\nobreak\hfil
\penalty50\hskip1em\null\nobreak\hfil\squareforqed
\parfillskip=0pt\finalhyphendemerits=0\endgraf}\fi}
\def\squareforqed{\hbox{\rlap{$\sqcap$}$\sqcup$}}
\def\pf{\par\noindent {\it Proof }}
\theoremstyle{plain}
\newtheorem{Theorem}{Theorem}
\newtheorem{Corollary}[Theorem]{Corollary}
\newtheorem{Lemma}[Theorem]{Lemma}
\newtheorem{Proposition}[Theorem]{Proposition}
\newtheorem{Fact}[Theorem]{Fact}
\theoremstyle{definition}
\newtheorem{Definition}[Theorem]{Definition}
\newtheorem{Problem}[Theorem]{Problem}
\newtheorem{Void}{}[section]
\theoremstyle{remark}
\newtheorem{Example}[Theorem]{Example}
\newtheorem{Remark}[Theorem]{Remark}
\newtheorem{Algorithm}[Theorem]{Algorithm}
\newtheorem{Procedure}[Theorem]{Procedure}
\newtheorem{History}[Theorem]{Historical Remark}
\newtheorem{Notation}[Theorem]{Notation}
\newtheorem{teo}{Remark$\times$teo}
        
\newcounter{saveq}
\newcommand{\alpheqn}{\setcounter{saveq}{\value{equation}}
\stepcounter{saveq}\setcounter{equation}{0}
\renewcommand{\theequation}
{\mbox{\arabic{saveq}-\alph{equation}}} 
}
\newcommand{\reseteqn}{\setcounter{equation}{\value{saveq}}%
\renewcommand{\theequation}{\arabic{equation}}}                 
\begin{abstract} We present here a new approach for computing Gr\"obner bases for bilateral modules over an effective ring. Our method is based on Weispfenning notion of restricted Gr\"obner bases and related multiplication.
\end{abstract}
For (commutative) polynomial rings ${\Bbb F}[X_1,\ldots, X_n]$ \cite{Bu1,Bu2,Bu3,BCrit} over a field, Gr\"obner bases are computed by an iterative application of Buchberger test/completion which
 states that {\em a basis  $F$ is Gr\"obner  if and only if each element in the set of all S-polyno\-mials
$$\left\{S( f_{\alpha'},   f_{\alpha}):= 
\frac{\lcm({\bf M}(f_{\alpha}), {\bf M}(f_{\alpha'}))}{{\bf M}(f_{\alpha})} 
 f_{\alpha}- \frac{\lcm({\bf M}(f_{\alpha}), {\bf M}(f_{\alpha'}))}{{\bf M}(f_{\alpha'})}  f_{\alpha'} : f_\alpha, f_{\alpha'}\in F\right\}$$
between two elements of $F$, reduces to 0}.

The same result holds for free monoid rings ${\Bbb F}\langle X_1,\ldots, X_n\rangle$  over a field,
even if the shape of the {\em matches} (S-polynomials) between two elements is more involved and, in general, between two elements there could even be {\em infinitely many} S-polynomials; of course, in this setting, there is no hope of termination. Anyway, there are classical techniques
 \cite{Pr2} producing a procedure which, receiving as input a finite generating set $F$ for
the module ${\Bbb I}(F)$, provided that ${\Bbb I}(F)$ has a finite  Gr\"obner  basis, halts returning such a finite  Gr\"obner basis.

In both cases, it is well known that Buchberger  test/completion  is definitely superseded in each honest survey of Buchberger Theory and (what is more important) in all available implementations, by 
the test/completion based on the lifting theorem \cite{M}: {\em a generating set $F$ is a Gr\"obner basis if and only if
each element in a minimal basis of the syzygies among the
leading monomials $\{{\bf M}(f_{\alpha}) : f_{\alpha}\in F\}$ lifts, via Buchberger reduction, 
to a syzygy among the elements of $F$}. 

The point is that the lifting theorem allowed Gebauer--M\"oller \cite{GM} to give more efficient criteria. 
Thus they  detect at least as many ``useless'' pairs 
as Buchberger's two criteria \cite{BCrit}, but {\em  they do not need to verify whether a pair satisfies the conditions required by the Second Criterion and thus they avoid the consequent bottleneck} needed for listing and ordering the  
S-pairs (in the commutative case they are $(\#F)^2$ while a careful informal analysis in that setting suggests that the S-pairs needed by Gebauer--M\"oller Criterion are $n\#F$).
Moreover, the flexibility of M\"oller lifting theorem approach - with respect to Buchberger S-pair test - allows the former to extend Buchberger theory  {\em verbatim} at least to (non commutative) monoid rings over PIRs. 

We can remark that Buchberger Theory and Algorithm for left (or right) ideals of 
monoid rings over PIRs essentially  repeats {\em verbatim} the same Theory and Algorithm
as the commutative case.

The same happens in the first class of twisted polynomial rings whose Buchberger Theory and Algorithm has been studied, {\em solvable polynomial rings} over a field \cite{KrW}: there the left case is obtained simply by reformulating  Buchberger test, while the bilateral case is solved via Kandri-Rody--- Weispfenning completion which essentially consists of a direct application of Spear's Theorem.

Later, Weispfenning studied an interesting class of rings,
${\Bbb Q}\langle x,Y\rangle/{\Bbb I}(Yx-x^eY),$ $e\in{\Bbb N}, e>1$ \cite{W}, \cite[IV.49.11,IV.50.13.6]{SPES},
and essentially applied the same kind of completion:
instead of the bilateral ideal
$${\Bbb I}_2 := \Span_{\Bbb Q}\left(x^aY^bfx^cY^d : (a,b,c,d)\in{\Bbb N}^4\right)$$
he considered the {\em restricted} ideal
$${\Bbb I}_W := \Span_{\Bbb Q}\left(x^afY^d : (a,d)\in{\Bbb N}^2\right).$$
Then he computed a {\em restricted Gr\"obner basis}  of it via Buchberger test and extended this restricted Gr\"obner basis to the
required  bilateral Gr\"obner basis via a direct application of Spear's Theorem.
The point is that, if we denote $\diamond$ the commutative multiplication
$$x^aY^d\diamond x^cY^b =x^{a+c}Y^{b+d}, (a,b,c,d)\in{\Bbb N}^4,$$
the computation of restricted Gr\"obner bases {\em verbatim} mimicks the commutative case as it was done for left ideals in the case of  solvable polynomial rings.

A Buchberger Theory for each effective ring 
$${\Cal A}={\Cal Q}/{\Cal I}, {\Cal Q} := {\Bbb D}\langle {\Bcc v}\sqcup{\Bcc V}\rangle,
{\Cal I} = {\Bbb I}_2(G),$$
where ${\Bbb D}$ is a PID and $G$ a Gr\"obner basis w.r.t. a suitable term ordering $<$,
has been recently proposed in \cite[IV.50]{SPES} (for an abridged survey see \cite{Bath}),
using the strength of M\"oller lifting theorem.

In this setting, denoting $G_0 := G\cap{\Bbb D}\langle {\Bcc v}\rangle$, we need to consider S-pairs among elements which essentially have the shape
\begin{itemize}
\item $a\omega f, f\in F, \omega\in\langle {\Bcc V}\rangle, a\in{\Bbb D}\langle{\Bcc v}\rangle/{\Bbb I}_2(G_0)$ in the left case, and
\item $a\lambda f b\rho, f\in F, \lambda,\rho\in\langle {\Bcc V}\rangle, a,b\in{\Bbb D}\langle {\Bcc v}\rangle/{\Bbb I}_2(G_0)$ in the bilateral case.
\end{itemize}
\bigskip

While reading the proofs of  \cite[IV]{SPES} the senior author realized a wrong description of the S-polynomials required by the bilateral lifting theorem in an example involving the Ore algebra
${\Bbb Z}[X,Y,Z]/{\Bbb I}(YX-2XY,ZX-3XZ,ZY-5XZ)$ \cite[IV.50.11.8]{SPES} which was therefore forced to remove; at the same time, however, the reading of the section devoted to Weispfenning ring suggested him how to formalize an intuition informally expressed in \cite{Roma}. Applying this approach to Ore algebras \cite{labOrE} the junior author  formalized the notion of Weispfenning multiplication $\diamond$ and realized that it allows to extend {\em verbatim} Buchberger First Criterion and, consequently, the algorithms based on Gebauer-M\"oller Criteria \cite{GM}, \cite[II.25.1]{SPES}.

This provides an alternative (and more efficient) approach for producing bilateral Gr\"obner bases, via the notion of restricted Gr\"obner bases, for which we have to apply the test to elements  having the shape
\begin{itemize}
\item $a\omega \diamond f, f\in F, \omega\in\langle {\Bcc V}\rangle, a\in{\Bbb D}\langle{\Bcc v}\rangle/{\Bbb I}_2(G_0)$
\end{itemize}
and for which Gebauer-M\"oller Criteria are available; once a bilateral Gr\"obner basis is thus produced a direct application of Spear Theorem is all one needs. 
\bigskip

In Sections 1-3 we discuss in detail our notion of  {\em effective ring}, {\em i.e.} a ring ${\Cal A}$ 
presented, accordingly the universal property of free monoid rings, as a quotient ${\Cal A}={\Cal Q}/{\Cal I}$ of
a free monoid ring  ${\Cal Q} := {\Bbb D}\langle {\Bcc v}\sqcup{\Bcc V}\rangle$ modulo a bilateral ideal ${\Cal I} = {\Bbb I}_2(G)$, presented in turn by
its Gr\"obner basis w.r.t. a suitable term ordering $<$.
Thus the ring ${\Cal A}$ turns out to be a left R-module over the effectively given ring 
$$R := {\Cal R}/{\Bbb I}_2(G_0),  {\Cal R} :={\Bbb D}\langle {\Bcc v}\rangle, G_0 := G\cap{\Cal R}.$$

In Section 4 we discuss the
pseudovaluation \cite{A}  which is naturally induced on  ${\Cal A}$ by the classical filtration/valuation
of ${\Cal Q}$  related with Buchberger Theory, so that in Section 5 we can import on ${\Cal A}$ the
notions and main properties of Gr\"obner bases, Gr\"obner presentation, normal forms.

At the same time after having introduced Weispfenning multiplication (Section~6), we can extend the same notions and properties (Section~7) to the case of restricted modules, proving a lifting theorem for them  (Section~8) 
and consequently listing the S-polyomials needed to test/completing a restricted basis   (Section~11); an adaptation  of Weispfenning Completion in this setting (Section~9), allows to produce, iteratively, a bilateral Gr\"obner basis from which a strong bilateral Gr\"obner basis can be easisly deduced (Section~12). 

Of course, in this setting it is well-known that there is no chance to hope for a terminating algorithm, unless the ring is noetherian and its representation is properly restricted; the classical approach consists in producing a procedure which terminates if and only if the module generated by a given finite basis
has a finite Gr\"obner basis which, in this case, is returned (Section~10).
\bigskip

The paper is completely self-contained and can be read without knowing \cite{SPES} and \cite{Bath}; it requires however a good knowledge of the 
classical papers on which is based the core of Buchberger Theory: the results by Buchberger  \cite{Bu1,Bu2,Bu3,BCrit}, Spear \cite{Sp}, Zacharias \cite{Z}, M\"oller \cite{M}, Gebauer-M\"oller \cite{GM}, Traverso\cite{TD, Sugar}, Weispfenning \cite{KrW,BW,W}, Pritchard \cite{Pr1,Pr2}, Apel \cite{A}.
\section{Effectiveness}

Given any set  ${\Bcc Z}$ and denoting $\langle {\Bcc Z}\rangle$ the monoid of all words over the alphabet ${\Bcc Z}$,
we can consider the 
free monoid ring ${\Cal Q} := {\Bbb D}\langle {\Bcc Z}\rangle$ of $\langle {\Bcc Z}\rangle$
over the principal ideal domain  ${\Bbb D}$ 
whose elements are the finite sums of ``monomials'' $c\tau, c\in{\Bbb D}, \tau\in\langle {\Bcc Z}\rangle$, and whose product is obtained by distributing the word concatenation  of $\langle {\Bcc Z}\rangle$ :
$$cx_1x_2\ldots x_m\cdot  dy_1\ldots y_n=cdx_1x_2\ldots x_my_1\ldots y_n \Forall c,d\in{\Bbb D}, x_i,y_j\in {\Bcc Z}.$$

The ring ${\sf Q} := {\Bbb Z}\langle {\Bcc Z}\rangle$ has the following universal property: any map ${\Bcc Z}\to{\sf A}$ over any ring  with identity ${\sf A}$ can be uniquely extended to a ring morphism 
${\sf Q} \to {\sf A}$. Therefore:

\begin{Fact}\label{50cF1} 
For a (not necessarily commutative) ring with identity ${\sf A}$, there is a (not necessarily finite nor necessarily countable) set ${\Bcc Z}$ 
and a projection
$\Pi : {\sf Q} := {\Bbb Z}\langle {\Bcc Z}\rangle \twoheadrightarrow {\sf A}$ so that, denoting
${\sf I}\subset
{\sf Q} = {\Bbb Z}\langle {\Bcc Z}\rangle$ the bilateral ideal ${\sf I} := \ker(\Pi)$, we have ${\sf A}={\sf Q}/{\sf I}$.
\end{Fact}

\begin{proof} It is sufficient to consider the set ${\Bcc Z} := {\sf A}$
and the identity map ${\Bcc Z} := {\sf A}\to{\sf A}$ in order to obtain the result by the 
universal property of ${\sf Q} := {\Bbb Z}\langle  {\sf A}\rangle$.
\end{proof}

Of course, each {\em commutative} ring ${\sf A}$ can be represented in a similar way as a quotient of the commutative polynomial ring ${\sf P} := {\Bbb Z}[{\Bcc Z}]$ modulo an ideal ${\sf I}$.

\bigskip

Let $R$ be a (not necessarily commutative) ring  with identity ${\bf 1} _R$ and 
${\Cal A}$ 
another (not necessarily commutative) ring 
with identity  ${\bf 1} _{\Cal A}$
which is a left module on $R$.

\begin{Definition}\cite{Zac}\label{EfGi} We
consider  ${\Cal A}$ to be {\em effectively given}  when we are given
\begin{itemize}
\renewcommand\labelitemi{\bf --}
\item a Zacharias \cite[II.26.1]{SPES} principal ideal domain ${\Bbb D}$ with {\em canonical representatives} \cite{Zac};
\item sets ${\Bcc v} := \{x_1,\ldots,x_j,\ldots\}$, ${\Bcc V} := \{X_1,\ldots,X_i,\ldots\}$, which are {\em  countable}, 
and 
\item ${\Bcc Z} := {\Bcc v}\sqcup{\Bcc V} = \{x_1,\ldots,x_j,\ldots,X_1,\ldots,X_i,\ldots\};$ 
\item rings 
${\Cal R} := {\Bbb D}\langle {\Bcc v}\rangle \subset
 {\Cal Q} := {\Bbb D}\langle {\Bcc Z}\rangle$;
\item projections $\pi :  {\Cal R} = {\Bbb D}\langle x_1,\ldots,x_j,\ldots\rangle
 \twoheadrightarrow R$ and
\item $\Pi : {\Cal Q} := {\Bbb D}\langle x_1,\ldots,x_j,\ldots,X_1,\ldots,X_i,\ldots\rangle \twoheadrightarrow {\Cal A}$ which satisfy
$$\Pi(x_j)=\pi(x_j){\bf 1} _{\Cal A}, \Forall x_j\in{\Bcc v},$$ so that 
$\Pi\left( {\Cal R}\right) =\left\{ r{\bf 1} _{\Cal A} :r\in R\right\}\subset{\Cal A}$.
\end{itemize}
Thus denoting
\begin{itemize}
\renewcommand\labelitemi{\bf --}
\item  ${\Cal I} := \ker(\Pi)\subset {\Cal Q}$ and
\item  $I := {\Cal I}\cap{\Cal R} =\ker(\pi)\subset{\Cal R}$,
\end{itemize} 
we have
${\Cal A}={\Cal Q}/{\Cal I}$ and $R={\Cal R}/I$; moreover we can wlog 
assume that $R\subset{\Cal A}$.

Further, when considering ${\Cal A}$ as effectively given in this way, we explicitly impose the Ore-like requirement that 
\begin{equation}\label{50cEqAlg}
X_ix_j\equiv\sum_{l=1}^{i} \pi(a_{lij})X_l+ \pi(a_{0ij})\bmod{{\Cal I}},
a_{lij}\in  {\Bbb D}\langle {\Bcc v}\rangle, 
\end{equation}
for all $X_i\in{\Bcc V}, x_j\in {\Bcc v}$.
\end{Definition}

\begin{Remark}\label{50cRRem}
\begin{enumerate}\

\item  
It is sufficient to consider the uncountable field of the reals ${\Bbb R}$, to understand that not necessarily each ring ${\Cal A}$ can be provided of a Buchberger Theory.

Essentially, our definition of an {\em effectively given} ring ${\Cal A}$ is a specialization of the one introduced (under the same name of {\em explizite-bekann}) for fields by van der Waerden \cite{vdW}; the difference is that the ablity of performing arithmetics in {\em endlichvielen Schritten} is granted here by the implicit assumption of knowing a Gr\"obner basis of ${\Cal I}$.

Moreover, in the commutative case, the recent result of \cite{W-W} which, following an old idea of Buchberger \cite{Bu4}, obtains a degree-bound evaluation for ideal membership test and canonical form computation by merging Grete Hermann's \cite{Her} and Dub\'e's \cite{Dub} bounds,
grants 
a representation of ${\Cal A}$ which even satisfies  Hermann's \cite[p.736]{Her} requirement of {\em an upper bound for the number of operations needed by the computation}.

If we are interested in polynomial rings with coefficients in ${\Bbb R}$ or in a ring of analytical functions (as in Riquiet-Janet Theory \cite{Jan1,Jan3,Pom}), since a given finite basis has a finite number of coefficients  
$c_i\in R$, the requirement that the data are effectively given essentially means that we need to provide the algebraically dependencies among such $c_i$.

For instance while the rings ${\Bbb Q}[\pi]$ and ${\Bbb Q}[e]$ can be considered effectively given as ${\Bbb Q}[v]$ within Kronecker's Model \cite[I.8.1-3.]{SPES}, the problem arises with ${\Bbb Q}[\pi,e]$: the 
Kronecker's Model ${\Bbb Q}[v_1,v_2]$ is valid provided that 
$\pi$ and $e$ are algebraically independent; potential 
algebraic dependencies generate an ideal ${\sf I}\subset{\Bbb Q}[v_1,v_2]$
and the ring can be considered effectively given under Definition~\ref{EfGi} 
only if such ideal is explicitly produced thus representing
${\Bbb Q}[\pi,e]$ as ${\Bbb Q}[v_1,v_2]/{\sf I}$; the point, of course, is that the status 
of algebraically dependency between $\pi$ and $e$ is still open.
\item The Ore-like requirement (\ref{50cEqAlg}),  which wants that no higher-indexed ``variable'' $X_l, l>i$, appears in the representation, in the left ${\Cal R}$-module  ${\Cal A}$, of a multiplication of a ``variable'' $X_i$ at  the right by a ``coefficient'' $x_j$, is necessary in order to avoid non-noetherian reductions.

In order to illustrate the r\^ole of condition~(\ref{50cEqAlg}), the most natural example is the free monoid ring ${\Bbb Z}\langle x,y\rangle$ which is naturally a left ${\Bbb Z}[x]$-module; a natural choice for the generating set $\langle\Pi({\Bcc V})\rangle = \Pi(\langle{\Bcc V}\rangle)$ is ${\Bcc V} =\{X_i, i\in{\Bbb N}\}, \Pi(X_i) = yx^i$ which gives, through the isomorphism $\Pi$, the equivalent representation ${\Bbb Z}\langle x,y\rangle\cong 
{\Bbb Z}[x]\langle {\Bcc V}\rangle$ and the projection
$$\Pi : {\Bbb Z}\langle x,X_0,X_1,\ldots\rangle \twoheadrightarrow {\Bbb Z}\langle x,y\rangle, \ker(\Pi) = \{X_ix-X_{i+1},  i\in{\Bbb N}\},$$
and in order to obtain ${\bf T}(X_ix-X_{i+1}) = X_ix$ we are forced to use the non-noetherian ordering
$X_1 >_V X_2 >_V \ldots >_V X_i >_V  \ldots$ on ${\Bcc V}$ which would require a related Hironaka Theory  \cite{Hir}.

Thus our definition considers ${\Bbb Z}\langle x,y\rangle$ as {\em not}
 effectively given as a left ${\Bbb Z}[x]$-module.
\qed\end{enumerate}\end{Remark}

For each $m\in{\Bbb N}$, we denote $\{{\bf e}_1,\ldots,{\bf e}_m\}$    
 the canonical basis  of the free ${\Cal Q}$-module ${\Cal Q}^m$,
whose basis as a left ${\Bbb D}$-module is the set of {\em terms}
$$\langle {\Bcc Z}\rangle^{(m)}:=\{t{\bf e}_i : t\in\langle {\Bcc Z}\rangle, 1\leq i\leq m\}.$$  


If we impose on $\langle {\Bcc Z}\rangle^{(m)}$ a term ordering $<$, then 
each $f\in {\Cal Q}^m$  has a unique representation as 
an ordered linear combination of terms $t\in\langle {\Bcc Z}\rangle^{(m)}$ with coefficients in ${\Bbb D}$:
$$f = \sum_{i=1}^s c(f,t_i) t_i : c(f,t_i) \in {\Bbb D}\setminus\{0\}, t_i \in \langle {\Bcc Z}\rangle^{(m)}, 
t_1 > \cdots > t_s.$$
The {\em support} of $f$ is the set $\supp(f)
:= \{t : c(f,t)\neq 0\};$ we further denote
${\bf T}(f) := t_1$
the {\em maximal term}\index{maximal!term} of $f$, $\lc(f)
:= c(f,t_1)$ its {\em leading coefficient}\index{leading!cofficient} and ${\bf M}(f) := c(f,t_1)t_1$ 
its {\em maximal monomial}\index{maximal!monomial}.

For a subset $G\subset{\Cal Q}^m$ of a module ${\Cal Q}^m$, ${\Bbb I}_L(G), {\Bbb I}_R(G),{\Bbb I}_2(G)$ denotes the left (resp. right, bilateral) module generated by $G$, the index being dropped when there is no need of specification; moreover  
${\bf T}\{G\}$ denotes the set
$${\bf T}\{G\}:=\{{\bf T}(f) : f\in {\sf I}\}\subset\langle {\Bcc Z}\rangle^{(m)}.$$
\section{Recalls on Zacharias rings and canonical representation}

Zacharias approach \cite{Z}  to Buchberger Theory consisted in remarking that,
 if each module ${\sf I}\subset R\langle {\Bcc Z}\rangle^m$ has a groebnerian
property, necessarily the same property must be satisfied at least by the
modules ${\sf I}\subset R^m\subset R\langle {\Bcc Z}\rangle^m$ and thus such property in $R$ is available and can be used to device a procedure granting the same property in $R\langle {\Bcc
Z}\rangle^m$.
The most elementary applications of Zacharias approach is the generalization  
(up to membership test and syzygy computation) of
the property of canonical forms from the case in which $R={\Bbb F}$ is a  
field to the general case: all we need is an effective notion of canonical
forms for modules in $R$. 

\begin{Definition}[Zacharias]  \cite{Z} \label{c46D1}
 A   ring $R$ is said to have
{\em  canonical representatives}  if
 there is an algorithm  which, given an element $c\in R^m$ and a  (left, bilateral, right)  module
 ${\sf J}\subset R^m$,
computes a {\em unique} element $\Rep(c,{\sf J})\in R^m$ such that
\begin{itemize}
\renewcommand\labelitemi{\bf --}
\item  $c-\Rep(c,{\sf J})\in{\sf J}$,
\item $\Rep(c,{\sf J}) = 0 \iff c\in{\sf J}$.
\end{itemize}
The set 
$$R^m\supset{\bf Zach}(R^m/{\sf J}) := \Rep({\sf J}) := \left\{\Rep(c,{\sf J}) :c\in R^m \right\}\cong R^m/{\sf J}$$
is called the {\em canonical Zacharias representation} of the module $R^m/{\sf J}$.
\qed\end{Definition}

Remark that, for each $c,d\in R^m$ and each module ${\sf J}\subset R^m$, we have
$$c-d\in{\sf J} \iff \Rep(c,{\sf J})=\Rep(d,{\sf J}).$$

\begin{Definition}  \cite{Z} (cf. \cite[II. Definition~{26.1.1}]{SPES})  \label{c46D3}
\renewcommand\theenumi{{\rm (\alph{enumi})}}
 A ring $R$ with identity is called 
 a (left) {\em Zacharias ring}\index{Zacharias!ring}\index{ring!Zacharias}
if it satisfies the following properties:
\begin{enumerate}
\item R is a noetherian ring;
\item  there is an algorithm which, for each 
$c \in R^m, \, C := \{c_1,\ldots c_t\} \subset R^m\setminus\{0\},$
 allows to decide whether 
$c\in {\Bbb I}_L(C)$  in which case it produces elements 
$d_i\in R :  c = \sum_{i=1}^t d_i c_i;$
\item there is an algorithm  which, given
$\{c_1,\ldots c_t\} \subset R^m\setminus\{0\},$ computes a finite set of
generators for 
the left syzygy R-module  $\left\{(d_1,\cdots,d_t) \in R^t :
\sum_{i=1}^t d_i c_i = 0\right\}$.
\end{enumerate}

Note that
 \cite{M}   for a  ring $R$ with identity which satisfies (a) and (b), (c) is equivalent to
\begin{enumerate}
\setcounter{enumi}{3}
\item there is an algorithm  which, given
$\{c_1,\ldots c_s\} \subset R^m\setminus\{0\},$ computes a finite basis of the
ideal
$${\Bbb I}_L(\{c_i : 1\leq i < s\}) : {\Bbb I}_L(c_s).$$
\end{enumerate}

If $R$ has canonical representatives,
we improve the computational assumptions
of Zacharias rings, requiring also  the following property:
\begin{enumerate}\setcounter{enumi}{4}
\renewcommand\theenumi{{\rm (\alph{enumi})}}
\item there is an algorithm  which, given an element $c\in R^m$ and a 
 left
module ${\sf J}\subset R^m$, computes the unique canonical representative
 $\Rep(c,{\sf J})$.
 \qed\end{enumerate}
\end{Definition} 

We can now precise our assumption on ${\Bbb D}$ requiring that it is a Zacharias PID with canonical representatives.

We begin by noting that when ${\Bbb D}={\Bbb Z}$, for each $m\in {\Bbb Z}$,  reasonable sets $A_m$
of the canonical
representatives of the residue classes of ${\Bbb Z}_m={\Bbb Z}/{\Bbb I}(m)$
are $$A_m = \{z\in {\Bbb Z} : -\frac{m}{2}<z\leq\frac{m}{2}\},
A_m = \{z\in {\Bbb Z} : 0<z\leq m\} \Or A_m = \{z\in {\Bbb Z} : 0\leq z< m\}.$$ 
In the general case we remark that, if we use Szekeres notation \cite{S}, \cite[IV.46.1.1.2]{SPES}, \cite{Zac} and denote
${\sf I}_\tau$ the left {\em Szekeres ideal}
$${\sf I}_\tau := \{\lc(f) : f\in{\sf I}, {\bf T}(f) = \tau\}\cup\{0\} = {\Bbb I}(c_\tau)\subset {\Bbb D}$$
and $c_\tau$ its {\em Szekeres generator},
for each module ${\sf I} \subset {\Cal Q}^m$  and each 
$\tau\in\langle {\Bcc Z}\rangle^{(m)}$,
we obtain

\begin{itemize}
\item the relation
$$\omega\mid\tau \then 
c_\tau\mid c_\omega,$$
for each $\tau,\omega\in{\bf T}\{{\sf I}\} := \{{\bf T}(f) : f\in {\sf I}\}\subset\langle {\Bcc Z}\rangle^{(m)};$
\item the partition 
$\langle {\Bcc Z}\rangle^{(m)} = {\bf L}({\sf I})\sqcup{\bf R}({\sf I})\sqcup{\bf N}({\sf I})$
of $\langle {\Bcc Z}\rangle^{(m)}$
where 
\begin{itemize}
\renewcommand\labelitemi{\mathbf --}
\item ${\bf N}({\sf I}) := \{\tau\in\langle {\Bcc Z}\rangle^{(m)} : {\sf I}_\tau = (0)\}$,
\item ${\bf L}({\sf I}) :=  \{\tau\in\langle {\Bcc Z}\rangle^{(m)} : {\sf I}_\tau = {\Bbb D}\}$,
\item ${\bf R}({\sf I}) := \left\{\tau\in\langle {\Bcc Z}\rangle^{(m)} : {\sf I}_\tau  \notin\bigl\{(0),{\Bbb D}\bigr\}\right\};$
\end{itemize}
\item  the {\em canonical Zacharias representation} 
\begin{eqnarray*}
{\Cal Q}^m\supset{\bf Zach}({\Cal Q}^m/{\sf I}) := \Rep({\sf I}) 
&=&
 \Bigl\{\Rep(c,{\sf I}) :c\in {\Cal Q}^m \Bigr\} 
\\ &:=& \bigoplus\limits_{\tau\in\langle {\Bcc Z}\rangle^{(m)}} \Rep({\sf I}_\tau)\tau
\\ &=& \bigoplus\limits_{\tau\in{\Cal T}^{(m)}}{\bf Zach}(R/{\sf I}_\tau)\tau
\cong {\Cal Q}^m/{\sf I}
\end{eqnarray*}
of the module ${\Cal Q}^m/{\sf I}.$
\end{itemize}

\section{Zacharias canonical representation of Effective Associative Rings}

If we fix 
\begin{itemize}
\renewcommand\labelitemi{\bf --}
\item a term-ordering $<$ on $\langle {\Bcc Z}\rangle$
\end{itemize}
we can assume  ${\Cal I}$ to be given via
\begin{itemize}
\renewcommand\labelitemi{\bf --}
\item its bilateral Gr\"obner basis $G$ w.r.t. $<$
\end{itemize}
and, if $<$ satisfies
\begin{equation}\label{50cEq<}
X_i>t \Forall t\in\langle {\Bcc v}\rangle \And X_i\in{\Bcc V},
\end{equation}
also $I$ is given via
\begin{itemize}
\renewcommand\labelitemi{\bf --}
\item its bilateral Gr\"obner basis $G_0 := G\cap{\Cal R}$ w.r.t. $<$. 
\end{itemize}

Since condition~(\ref{50cEqAlg}) implies that, for each $X_i\in{\Bcc V}, x_j\in {\Bcc v}$, 
$$f_{ij} := X_ix_j-\sum_{l=1}^{i} a_{lij}X_l- a_{0ij}\in{\Cal I}\subset{\Cal Q},$$
if we further require that $<$ satisfies
\begin{equation}\label{50cEq<>}
X_i x_j={\bf T}(f_{ij}) \Forall X_i\in{\Bcc V}, x_j\in {\Bcc v},
\end{equation}
and denote $C := \{f_{ij} : X_i\in{\Bcc V}, x_j\in {\Bcc v}\}$ we have
\begin{itemize}
\item $G_0\sqcup C\subset G$,
\item  ${\Cal A}$  is generated  as $R$-module by $\Pi(\langle{\Bcc V}\rangle)$
 and,
 \item   as ${\Bbb D}$-module, by a subset of
 $\Bigl\{
\upsilon\omega :\upsilon\in\langle {\Bcc v}\rangle, \omega\in\langle {\Bcc
V}\rangle\Bigr\}.$ 
\end{itemize}

Thus, using Szekeres notation and setting $A_{c_{\tau}} :={\Bbb D}/{\Cal I}_\tau$ for each $\tau\in\langle {\Bcc Z}\rangle$, 
${\Cal A}$ can be described via its Zacharias canonical representation  w.r.t. $<$ as
\begin{equation}
{\Cal A} = {\Cal Q}/{\Cal I} 
\cong
\bigoplus\limits_{\omega\in\langle {\Bcc V}\rangle} \left(\bigoplus\limits_{\upsilon\in\langle {\Bcc v}\rangle} A_{c_{\upsilon\omega}} \upsilon\right)\omega =: {\bf Zach}_<({\Cal A})
\subset{\Cal Q}.
\end{equation}   
 
\begin{Example} W.r.t. the ideal ${\sf I} := {\Bbb I}(2X,3Y)\in{\Bbb Z}[X,Y]$ whose strong Gr\"obner basis is $\{2X,3Y, XY\}$, the ring 
$${\Cal A} := {\Bbb Z}[X,Y]/{\sf I} \cong {\Bbb Z}\langle X,Y\rangle/{\Bbb I}_2(2X,3Y, XY,YX)$$ has the 
canonical representation
$${\Cal A} \cong {\Bbb Z} + {\Bbb Z}_2[X] X + {\Bbb Z}_3[Y] Y;$$
thus the underlying ${\Bbb Z}$-module has the structure
$${\Cal A}  \cong 
{\Bbb Z} \oplus \left(\bigoplus\limits_{i\in{\Bbb N}\setminus\{0\}}{\Bbb Z}_2\right)  \oplus \left(\bigoplus\limits_{i\in{\Bbb N}\setminus\{0\}}{\Bbb Z}_3\right)$$
and the ring structure is defined by
$$\left(a, \ldots d_i,\ldots g_i,\ldots\right)\star
\left(b, \ldots e_i,\ldots h_i,\ldots\right)=
\left(c, \ldots f_i, \ldots, l_i,\ldots\right)$$
where $a,b,c\in{\Bbb Z}, d_i,e_i,f_i\in{\Bbb Z}_2\cong\{0,1\}, g_i,h_i,l_i\in{\Bbb Z}_3\cong\{-1,0,1\}$
and
\begin{eqnarray*}
c &:=& ab, \\
 f_i &:=& \pi_2(a)e_i+\sum_{j=1}^{i-1}d_je_{i-j}+d_i\pi_2(b), i\in{\Bbb N}\setminus\{0\}, \\
l_i &:=& \pi_3(a)h_i+\sum_{j=1}^{i-1}g_jh_{i-j}+g_i\pi_3(b), i\in{\Bbb N}\setminus\{0\}.
\end{eqnarray*}
\qed\end{Example}

If we further consider, for each $\omega\in\langle {\Bcc V}\rangle$, 
the left Szekeres ideal  
$${\Cal I}_\omega := \left\{r\in{\Cal R} : \exists h\in{\Cal Q}, {\bf T}(h)<\omega, r\omega+h\in{\Cal I}\right\}
\supset I = {\Cal I}\cap{\Cal R}$$
and the ring 
$R_\omega = {\Cal R}/{\Cal I}_\omega,$ 
having the Zacharias canonical representation
$${\bf Zach}_<(R_\omega) \cong  \bigoplus\limits_{\upsilon\in\langle {\Bcc v}\rangle} A_{c_{\upsilon\omega}} \upsilon\subset{\Cal R}$$ 
we obtain 
{
$${\bf Zach}_<( {\Cal R}/{\Cal I}_\omega)\subset{\bf Zach}_<( {\Cal R}/I)=
{\bf Zach}_<(R)\subset{\Cal R}$$}
and
\begin{equation}\label{50cEqZac2}
{\Cal A} \cong 
\bigoplus\limits_{\omega\in\langle {\Bcc V}\rangle} \left(\bigoplus\limits_{\upsilon\in\langle {\Bcc v}\rangle} A_{c_{\upsilon\omega}} \upsilon\right)\omega \cong  
\bigoplus\limits_{\omega\in\langle {\Bcc V}\rangle} R_\omega \omega\subset {\Cal R}\langle {\Bcc V}\rangle={\Cal Q}.
\end{equation}   

More precisely, denoting
\begin{itemize}
\renewcommand\labelitemi{\bf --}
\item ${\bf N}({\Cal I}) :=   \{\omega\in\langle{\Bcc V}\rangle : {\Cal I}_\omega = I\}$,
\item ${\bf L}({\Cal I}) := \{\omega\in\langle{\Bcc V}\rangle : {\Cal I}_\omega = R\}$,
\item ${\bf R}({\Cal I}) := \left\{\omega\in\langle{\Bcc V}\rangle : {\Cal I}_\omega\notin\left\{I,R\right\}\right\} $
\end{itemize}
we have the partition
$\langle{\Bcc V}\rangle = {\bf L}({\Cal I})\sqcup{\bf R}({\Cal I})\sqcup{\bf N}({\Cal I})$ 
and, denoting
\begin{itemize}
\renewcommand\labelitemi{\bf --}
\item ${\Cal B} ={\bf R}({\Cal I})\sqcup{\bf N}({\Cal I}) = 
\langle{\Bcc V}\rangle\setminus{\bf L}({\Cal I})\subset\langle{\Bcc V}\rangle$,
\end{itemize}
we obtain
\begin{enumerate}
\item ${\Cal B}\subset\langle{\Bcc V}\rangle$ is an order module {\em i.e.} $\lambda\tau\rho\in{\Cal B}\then\tau\in{\Cal B}$ for each $\lambda,\tau,\rho\in\langle{\Bcc V}\rangle;$
\item ${\Cal A}$ is both a left ${\Cal R}$-module and a left $R$-module with generating set ${\Cal B}$.
\end{enumerate}

Thus, each element $f\in{\Cal A}$ is uniquely represented via its canonical representation w.r.t. $<$
$$
{\bf Rep}(f,{\Cal J}) = 
\sum_{\omega\in{\Cal B}} a_\omega \omega\in{\bf Zach}_<({\Cal A})
$$
where, using the present notation, each
$$a_\omega = \sum_{\upsilon\in\langle {\Bcc v}\rangle} b_{\upsilon\omega} \upsilon\in{\bf Zach}_<(R_\omega)
$$
is the canonical representation of an element of the module ${\Cal R}/{\Cal I}_\omega$ and each $b_{\upsilon\omega}\in A_{c_{\upsilon\omega}}$
is the canonical representation of an element of the ring 
$A_{c_{\upsilon\omega}}:={\Bbb D}/{\Bbb I}({c_{\upsilon\omega}})={\Bbb D}/{\Bbb I}_{\upsilon\omega}$; we will   identify the elements in ${\Cal A}$, $R_\omega$ and 
$A_{c_{\upsilon\omega}}$ with their representatives.

\begin{Example}\label{50cExEx}
For  
${\Cal Q} = {\Bbb Z}\langle x_1,x_2,X_1\rangle,$ 
$$G_0 = \{x_2x_1\}, C = \{X_1x_1-x_2X_1, X_1x_2-x_1X_1\},
{\Cal I} = {\Bbb I}_2(G_0\cup C), {\Cal A} = {\Cal Q}/{\Cal I},$$
a minimal Gr\"obner basis of ${\Cal I}$ is
$G_0\cup C \cup \{x_1x_2^{i+1}X_1, i\in{\Bbb N}\},$
since we have
$$x_1x_2X_1=X_1\star x_2x_1-(X_1x_2-x_1X_1)\star x_1-x_1\star(X_1x_1-x_2X_1)$$ and, for $i\geq 1$
$$x_1x_2^{i+1}X_1=x_1x_2^{i}X_1\star x_1-x_1x_2^{i}\star(X_1x_1-x_2X_1).$$

We therefore have
$$R = {\Bbb Z}\langle x_1,x_2 \rangle/{\Bbb I}(x_2x_1), {\bf Zach}_<(R) = \Span_ {\Bbb Z}\{x_1^ix_2^j : (i,j)\in{\Bbb N}^2\},$$
and, denoting $R_l := R_{X_1^l}, {\Cal I}_{l} := {\Cal I}_{X_1^l}$ for each $l$,
we have,  for $l\geq 1$,
$${\Cal I}_l = {\Bbb I}_L(x_2x_1,x_1x_2^{i+1},i\in{\Bbb N}),
R_l = {\Bbb Z}\langle x_1,x_2 \rangle/{\Cal I}_l \cong  {\Bbb Z}[x_1,x_2]/{\Bbb I}(x_1x_2)$$
so that  ${\bf Zach}_<(R_l) =  \Span_ {\Bbb Z}\{x_1^i, x_2^j : i,j\in{\Bbb N}\}$ 
and 
$${\bf Zach}_<({\Cal A}) 
= {\Bbb Z}[x_1,x_2]\bigoplus\left(\bigoplus\limits_{l\geq 1} {\Bbb Z}[x_1,x_2]/{\Bbb I}(x_1x_2)X_1^l\right)$$
so that a generic element of ${\bf Zach}_<({\Cal A})$ has the form
$$f(x_1,x_2) = a(x_1,x_2) + \sum_{l>0} \left(b_l + c_l(x_1)+d_l(x_2)\right)X_1^l$$ with $a\in{\Bbb Z}[x_1,x_2], b_l\in{\Bbb Z},c_l\in{\Bbb Z}[x_1],d_l\in{\Bbb Z}[x_2],c_l(0)=d_l(0)=0$ and the related left $R$-algebra structure is defined by
\begin{eqnarray*}
x_1^{i+1} f(x_1,x_2)&=&x_1^{i+1}a(x_1,x_2)+\sum_{l>0} \left(b_lx_1^{i+1} + c_l(x_1)x_1^{i+1}\right)X_1^l,
\\
x_2^{j+1} f(x_1,x_2)&=&x_2^{j+1}a(0,x_2)+\sum_{l>0} \left(b_lx_2^{j+1} + d_l(x_2)x_2^{j+1}\right)X_1^l,
\\
x_1^{i+1}x_2^{j+1} f(x_1,x_2)&=&x_1^{i+1}x_2^{j+1}a(0,x_2).
\end{eqnarray*}\qed\end{Example}

\begin{Remark}\

\begin{enumerate}
\item  We must stress  that all inclusions --- 
$A_{c_{\upsilon\omega}}\subset{\Bbb D}$, 
${\bf Zach}_<(R_\omega)\subset{\Cal R} = {\Bbb D}\langle {\Bcc v}\rangle,$
${\bf Zach}_<({\Cal A})\subset {\Cal R}[{\Cal B}]\subset{\Cal R}\langle{\Bcc V}\rangle$ ---
must be understood as {\em set} inclusions only and do not preserve the module structure and
the notation ${\Cal R}\langle{\Bcc V}\rangle$  does not denote the canonical monoid ring but, as the notation ${\Cal R}[{\Cal B}]$, only the underlying free left ${\Cal R}$-modules with bases 
$\langle{\Bcc V}\rangle$ and  ${\Cal B}$.
\item Note that   Zacharias' approach holds for any effective unitary ring $R$ with canonical representations; thus of course the r\^ole of ${\Bbb D}$  can be assumed on one side by each effectively given domain/field, on the other side by, say, ${\Bbb D}({\Bcc x})$, ${\Bbb Q}({\Bcc x})$,\ldots.
Actually, if we are interested in polynomial rings with coefficients in ${\Bbb R}$ or in a ring of analytical functions, since a given finite basis has a finite number of coefficients  
$c_i\in R$, the requirement that the data are effectively given essentially means that we need to provide the algebraically dependencies among such $c_i$ (compare Remark~\ref{50cRRem}.1.).
\item Condition~(\ref{50cEqAlg}), restricting the choice of $<$ to a term-ordering satisfying
 Equation~(\ref{50cEq<>}),
grants that, for each $i,j,$ $X_ix_j\in{\bf T}\{{\Cal I}\}$ and thus that 
$C\subset G$;
moreover, since there is no possible match among the leading terms 
$\{X_ix_j : X_i\in{\Bcc V}, x_j\in {\Bcc v}\}$, 
it also grants that, in ${\Cal Q}$ and under $<$, $C$ is a bilateral  Gr\"obner basis of the ideal
${\Bbb I}_2(C)$ it generates.

Since  there are the obvious matches 
$$\left\{{\bf T}(f_{ij})\ast \tau - X_i\ast x_j\tau :  X_i\in {\Bcc V}, x_j\tau\in{\bf T}\{G_0\}\right\}$$
in general we cannot expect that
$G_0\cup C$  is a bilateral  Gr\"obner basis of the ideal
${\Bbb I}_2(G_0\cup C)$ it generates; this in turn implies that 
as left $R$-module, ${\Cal Q}/{\Bbb I}_2(G_0\cup C)$ is {\em not} necessarily free (see Example~\ref{50cExEx}).
\item  In the next sections we will discuss expressions
$$f =\sum_{l=1}^\mu a_l \lambda_l \star g_{l}\star   b_l \rho_l :
\lambda_l,\rho_l\in{\Cal B}, a_l\in R_{\lambda_l}\setminus\{0\},  b_l\in R_{\rho_l}\setminus\{0\}, g_l \in B$$ where $f\in M$ is an element and $B\subset M$ is a basis of a bilateral ${\Cal A}$-module  $M$.
Each element $a_l\in R_{\lambda_l}\setminus\{0\}$ is  to be considered either
\begin{itemize}
\item as any non-zero element in a residue class  modulo the left ideal ${\Cal I}_{\lambda_l}$ in the ring
${\Cal R} = {\Bbb Z}\langle{\Bcc v}\rangle$ or
\item as the Zacharias canonical representation of such residue class in the set 
${\bf Zach}_<(R_{\lambda_l})\subset{\bf Zach}_<(R)\subset{\Cal R}$, or even
\item as any non-zero element in a residue class  modulo the left ideal $\pi({\Cal I}_{\lambda_l})$ in the ring
$R$ by simply identifying $R$ with its Zacharias canonical representation ${\bf Zach}_<(R).$
\end{itemize}
Consequently each element $a_l \lambda_l$ represents a ``monomial" in ${\Cal A}$ where the coefficient $a_l$ can be interpreted either in $R$ or in ${\Cal R}$ but in both cases represents a residue class or its canonical representation.

As a consequence, in all setting in which ${\Cal A}$ is mainly considered as a left $R$-module, we choose of writing 
$a_l\in R\setminus\{0\}$.
\item Each free  ${\Cal A}$-module ${\Cal A}^m, m\in{\Bbb N}$, 
 -- the canonical basis of which will be denoted by $\{{\bf e}_1,\ldots,{\bf e}_m\}$ --  is an  $R$-module with basis the set of the  {\em terms}
$${\Cal B}^{(m)} := \{ t{\bf e}_i : t\in{\Cal B}, 1\leq i\leq m\}$$
and 
the projection $\Pi : {\Cal Q} := {\Bbb D}\langle {\Bcc Z}\rangle \twoheadrightarrow {\Cal A}$,
${\Cal I} := \ker(\Pi)$, ${\Cal A}={\Cal Q}/{\Cal I}$, extends to each canonical projection, still denoted $\Pi$,
$$\Pi : {\Cal Q}^m \twoheadrightarrow {\Cal A}^m, \ker(\Pi) = {\Cal I}^m
={\Bbb I}_2(G^{(m)})$$ where $G$ is the Gr\"obner basis  w.r.t. $<$ of
${\Cal I}$ and $G^{(m)} := \{g{\bf e}_i,g\in G, 1\leq i \leq m\}$
 is the Gr\"obner basis of ${\Cal I}^m$ w.r.t. any term-ordering on $\langle{\Bcc Z}\rangle^{(m)}$ --- which we still denote  $<$ with a slight abuse of notation --- satisfying, for each $t_1,t_2\in\langle{\Bcc Z}\rangle, \tau_1,\tau_2\in\langle{\Bcc Z}\rangle^{(m)}$,
$$t_1\leq t_2, \tau_1\leq \tau_2\then t_1\tau_1\leq  t_2\tau_2, \tau_1t_1\leq \tau_2 t_2 .$$
\qed\end{enumerate}
\end{Remark}

In connection with the choice of the order module
$${\Cal B} ={\bf R}({\Cal I})\sqcup{\bf N}({\Cal I}) = 
\langle{\Bcc V}\rangle\setminus{\bf L}({\Cal I})\subset\langle{\Bcc V}\rangle$$
as  module basis of ${\Cal A}$,
Spear's Theorem \cite[IV.50.6.3]{SPES} suggests to consider it well-ordered by the same term-ordering $<$ on $\langle{\Bcc Z}\rangle$ which we have used for providing the Zacharias representation 
 of ${\Cal A}$ discussed above
 and which in particular satisfies Equations~(\ref{50cEq<}) and~(\ref{50cEq<>}).
In fact, in our setting Spear  states  that, for any module ${\sf M}\subset{\Cal A}^m$, denoting ${\sf M}' : =\Pi^{-1}({\sf M}) = {\sf M}+{\Cal I}^m$, we have
\begin{enumerate}
\item if $F$ is a reduced Gr\"obner basis of ${\sf M}'$, then 
$$\{g\in F : g = \Pi(g)\} =
\{\Pi(g) : g\in F, {\bf T}(g)\in{\Cal B}^{(m)}\} = F\cap{\bf Zach}_<({\Cal A})^m$$
is a Gr\"obner basis of ${\sf M}$;
\item if $F\subset{\bf Zach}_<({\Cal A})^m$ -- so that in particular $\Pi(f)=f$ for each $f\in F$ --
 is  the Gr\"obner basis of ${\sf M}$, then $F\sqcup G^{(m)}$ is a   Gr\"obner basis of ${\sf M}'$.
 \end{enumerate}

Thus, w.r.t. a term-ordering $<$  satisfiying Equations~(\ref{50cEq<}) and~(\ref{50cEq<>}),
each non-zero element $f\in{\Cal A}^{(m)}$ has its canonical representation
$$f := \sum_{j=1}^s c(f,t_j{\bf e}_{\iota_j}) t_j{\bf e}_{\iota_j}\in{\bf Zach}_<({\Cal A})^m, 
t_j\in{\Cal B},c(f,t_j{\bf e}_{\iota_j})\in R_{t_j}\setminus\{0\},
1\leq \iota_j\leq m,$$
with $t_1{\bf e}_{\iota_1} >  t_2{\bf e}_{\iota_{2}}>\cdots> t_s{\bf e}_{\iota_{s}}$
and we denote,
$\supp(f)
:= \{t_j{\bf e}_{\iota_j} : 1\leq j\leq m\}$ the {\em support} of $f$,
${\bf T}_<(f) := t_1{\bf e}_{\iota_1}$
its {\em maximal term}, $\lc_<(f)
:= c(f,t_1{\bf e}_{\iota_1})$ its {\em leading coefficient} and ${\bf M}_<(f) :=
c(f,t_1{\bf e}_{\iota_1})t_1{\bf e}_{\iota_1}$  its {\em maximal monomial}.

If we denote, following  \cite{R2,R3}, 
${\sf M}({\Cal A}^m) := \{ct{\bf e}_i \, \vert \, t\in{\Cal B},c\in R_t\setminus\{0\},  1\leq i\leq m\},$
the unique finite representation above
can be reformulated 
$$f = \sum_{\tau\in\supp(f)} m_\tau, \, m_\tau = c(f,\tau)\tau$$
as a sum of elements of the {\em monomial set} ${\sf M}({\Cal A}^m).$

\bigskip
These notions heavily depend on Zacharias representation which in turn depends on the 
term-ordering $<$ we have fixed on $\langle{\Bcc V}\rangle$.

This has an unexpected advantage:  already in the case of {\em semigroup rings} \cite{Ros,MR1,MR2}
${\Cal A} = R[{\sf S}]$,
an elementary adaptation of Buchberger Theory (which would suggest to set ${\Cal B} := {\sf S}$) is impossible since  ${\sf S}$ does not possess a semigroup ordering. The paradoxical solution consists  \cite{MR1,MR2}, or at least can be interpreted as \cite[IV.50.13.5]{SPES} considering ${\sf S} := {\Cal B}$ not as a semigroup but as a subset of a proper free semigroup $\langle{\Bcc V}\rangle\supset{\Cal B}$ and, via Spear's Theorem, import to ${\Cal A}$ the   
{\em natural $\langle{\Bcc V}\rangle$-pseudovaluation}
$${\bf T}(\cdot) : {\Cal A}^m \mapsto {\Cal B}^{(m)} : f \to {\bf T}(f)$$
  of $R\langle{\Bcc V}\rangle$.

The general solution, thus,  consists into applying the classical filtration/valuation interpretation of Buchberger Theory \cite{Sw,7,A,MoSw} and to impose on ${\Cal Q}$ a $\Gamma$-pseudova\-luation
$$ {\bf T}(\cdot) : {\Cal A}^m \mapsto {\Cal B}^{(m)} \subset \Gamma^{(m)} : f \to {\bf T}(f)$$
where the semigroup $(\Gamma,\circ), {\Cal B}\subset\Gamma\subset\langle{\Bcc V}\rangle$,
is properly chosen on the basis of the structural properties of the relation ideal ${\Cal I}$
in order to obtain a smoother arithmetics of the associated graded ring ${\Cal G}:=G({\Cal A})$. 

\section{Apel: pseudovaluation}
Denote, for a semigroup $(\Gamma,\circ)$,
 $\Gamma^{(u)}$ the sets
$$\Gamma^{(u)} := \{\gamma e_i, \gamma\in\Gamma, 1\leq i \leq u\}, u\in{\Bbb N},$$
endowed with no operation except the natural action of $\Gamma$
$$\Gamma\times\Gamma^{(u)}\times\Gamma \to \Gamma^{(u)}: 
 (\delta_l,\gamma,\delta_r) \mapsto \delta_l\circ\gamma\circ\delta_r, 
 \Forall \delta_l,\delta_r\in\Gamma, \gamma\in\Gamma^{(u)}.$$

\begin{Definition} If $(\Gamma,\circ)$ is a  semigroup, a ring $ {\Cal A}$ is called a {\em
$\Gamma$-graded ring}\index{graded!ring} if there is a family of subgroups $\{{\Cal A}_\gamma : \gamma\in\Gamma\}$ such that 
\begin{itemize}
\item ${\Cal A} = \bigoplus_{\gamma\in \Gamma} {\Cal A}_\gamma$,
\item ${\Cal A}_\delta{\Cal A}_\gamma  \subset {\Cal A}_{\delta\circ\gamma}$ for any $\delta,\gamma\in \Gamma.$
\end{itemize}

A right ${\Cal A}$-module $M$ of a $\Gamma$-graded ring ${\Cal A}$ is called a {\em $\Gamma^{(u)}$-graded
${\Cal A}$-module}\index{graded!module} if there is a family of subgroups $\{M_\gamma : \gamma\in \Gamma^{(u)}\}$
such that 
\begin{itemize}
\item $M = \bigoplus_{\gamma\in \Gamma^{(u)}} M_\gamma,$
\item $M_\gamma{\Cal A}_\delta \subset M_{\gamma\circ\delta}$ for any $\delta\in \Gamma, \gamma\in \Gamma^{(u)}.$
\end{itemize}

Given two  $\Gamma^{(u)}$-graded right ${\Cal A}$-modules $M,N$,
by a {\em $\Gamma$-graded morphism\index{graded!morphism}  $\phi: M \to N$ of degree $\delta\in\Gamma$}  we shall mean a morphism such that $\Phi(M_\gamma)\subset N_{\gamma\circ\delta}$ for each $\gamma\in \Gamma^{(u)}.$

An ${\Cal A}$-bimodule $M$ of a $\Gamma$-graded ring ${\Cal A}$ is called a {\em $\Gamma^{(u)}$-graded
${\Cal A}$-bimodule} if there is a family of subgroups $\{M_\gamma : \gamma\in \Gamma^{(u)}\}$
such that 
\begin{itemize}
\item $M = \bigoplus_{\gamma\in \Gamma^{(u)}} M_\gamma,$
\item ${\Cal A}_\delta M_\gamma \subset M_{\delta\circ\gamma}$ and
$M_\gamma {\Cal A}_\delta \subset M_{\gamma\circ\delta}$ for any $\delta\in \Gamma, \gamma\in \Gamma^{(u)}.$
\end{itemize}

Given two  $\Gamma^{(u)}$-graded  ${\Cal A}$-bimodules $M,N$ by a {\em $\Gamma$-graded morphism $\phi: M \to N$ of degree $(\delta_l,\delta_r)\in\Gamma^2$}, we shall mean a morphism such that $\Phi(M_\gamma)\subset N_{\delta_l\circ\gamma\circ\delta_r}$ for each 
$\gamma\in \Gamma^{(u)}.$

Each element $x\in M_\gamma$ is called {\em homogeneous} of degree $\gamma\in\Gamma^{(u)}$.

Each element $x\in M$ can be uniquely represented as a finite sum 
$x := \sum_{\gamma\in \Gamma^{(u)}} x_\gamma$ where $x_\gamma\in M_\gamma$ and $\{\gamma :
x_\gamma \neq 0\}$ is finite; each such element $x_\gamma$ is called a {\em homogeneous
component}  of degree $\gamma.$
\qed\end{Definition}

\begin{Definition}[Apel] \cite{A} Let  $(\Gamma,\circ)$ be a semigroup
 well-ordered by a semigroup ordering $<$, ${\Cal A}$ a  ring which is a left $R$-module over a subring $R\subset{\Cal A}$ 
and $M$  an  ${\Cal A}$-module.

A {\em $\Gamma$-pseudovaluation}\index{pseudovaluation}
is a function $v: {\Cal A}\setminus\{0\} \mapsto \Gamma$ such that,  for each $a_1,a_2\in{\Cal A}\setminus\{0\},$
\begin{enumerate}
\item $v(a_1-a_2) \leq \max(v(a_1),v(a_2))$,
\item $v(a_1a_2) \leq v(a_1)\circ v(a_2)$,
\item $v(r) = {\bf 1}_\Gamma$ for each $r\in R\subset{\Cal A}$.
\end{enumerate}

Impose now on $\Gamma^{(u)}$
 a well-ordering, denoted, with a slight abuse of notation also $<$,
satisfying, for each 
$\delta_1,\delta_2\in \Gamma, \gamma_1, \gamma_2\in \Gamma^{(u)}$
$$\delta_1\leq \delta_2, \gamma_1\leq \gamma_2\then \delta_1\circ \gamma_1\leq  \delta_2\circ \gamma_2, 
\gamma_1\circ \delta_1\leq  \gamma_2\circ \delta_2.$$

A function $w: M\setminus\{0\} \mapsto \Gamma^{(u)}$ is said a {\em $v$-compatible
 $\Gamma^{(u)}$-pseudo\-valuation} on $M$ if it satisfies, for each
$a\in{\Cal A}\setminus\{0\}$ and each $m,m_1,m_2\in M\setminus\{0\},$ 
\begin{enumerate} \setcounter{enumi}{3}
\item $w(m_1-m_2) \leq \max(w(m_1),w(m_2)),$
\item $w(a m) \leq v(a)\circ w(m)$ and $w(m a) \leq w(m)\circ v(a).$
\end{enumerate}
\qed\end{Definition}

\begin{Notation} ({\em Cf.} \cite[II.Definition 24.6.5]{SPES} \label{ApNot} Given a  semigroup  $(\Gamma,\circ)$ 
well-ordered by a semigroup ordering $<$, a  ring ${\Cal A}$ which is a left $R$-module over a subring $R\subset{\Cal A}$, a $\Gamma$-pseudovaluation $v: {\Cal A}\setminus\{0\} \mapsto \Gamma$, an  ${\Cal A}$-bimodule $M$  and  a  $v$-compatible $\Gamma^{(u)}$-pseudovaluation $w: M\setminus\{0\} \mapsto \Gamma^{(u)}$  write
\begin{itemize}
\item $F_\gamma(M) :=\{m\in M : w(m) \leq \gamma\} \cup\{0\} \subset M,$ for each
$\gamma\in\Gamma^{(u)};$  
\item $V_\gamma(M) :=\{m\in M : w(m) < \gamma\} \cup\{0\} \subset M,$ for each
$\gamma\in\Gamma^{(u)};$
 \item $G_\gamma(M) := F_\gamma(M)/V_\gamma(M),$ for each
$\gamma\in\Gamma^{(u)};$  
\item $G(M) := \bigoplus_{\gamma\in \Gamma^{(u)}} G_\gamma(M)$.
\item ${\Cal L} : M \mapsto G(M)$  is the map such that, for each 
$m\in M, m \neq 0, {\Cal L}(m)$ denotes the residue class of $m \bmod V_{w(m)}(M)$ and ${\Cal
L}(0) = 0.$ 
\qed\end{itemize}
\end{Notation}

\begin{Definition} With the present notation,  we define
\begin{itemize}
\item the {\em associated graded ring} of ${\Cal A}$ the left $R$-module $G({\Cal A})$  which is a $\Gamma$-graded ring,
and 
\item the {\em associated graded module}\index{associated
graded!module} of $M$
the left $R$-module $G(M)$, which is a $\Gamma^{(u)}$-graded $G({\Cal A})$-module.
\qed\end{itemize}\end{Definition}

As we have remarked above, 
when the ring ${\Cal A}$ is explicitly given via the Zacharias representation (\ref{50cEqZac2}) we cannot use
 the function 
$$ {\bf T}(\cdot) : {\Cal A} \mapsto {\Cal B} : f \to {\bf T}(f)$$
as a  natural pseudovaluation because, in general, either ${\Cal B}$ is not a semigroup or, at least, $<$ is not a semigroup ordering on it. 

Thus we consider 
a semigroup $\Gamma$, ${\Cal B} \subset \Gamma \subset \langle {\Bcc V}\rangle,$ such that the restriction of $<$ on 
$\Gamma$ is a semigroup ordering.
In this way, the function 
$${\bf T}(\cdot) : {\Cal A} \mapsto {\Cal B} \subset \Gamma : f \to {\bf T}(f)$$
is a $\Gamma$-pseudovaluation, which we will call its {\em natural $\Gamma$-pseudovaluation} 
and the free  ${\Cal A}$-module ${\Cal A}^m$ has the  {\em natural  ${\bf T}(\cdot)$-compatible
pseudovaluation}
$$ {\bf T}(\cdot) : {\Cal A}^m \mapsto {\Cal B}^{(m)} \subset \Gamma^{(m)} : f \to {\bf T}(f).$$

Under these natural pseudovaluations, we have
\begin{itemize}
\item $G_\delta({\Cal A}) \cong R_\delta$ for each $\delta\in{\Cal B}$ and
\item $G_\delta({\Cal A}) = \{0\}$ for each $\delta\in \Gamma\setminus {\Cal B}$;
\item $G({\Cal A})$ and ${\Cal A}$ coincide {\em as subsets, (but not as rings nor as $R$-modules)}  and both have the Zacharias representation stated in (\ref{50cEqZac2});
\item $G_\gamma({\Cal A}^m) \cong R_\delta$ for each $\gamma=\delta{\bf e}_i\in{\Cal B}^{(m)}$ and
\item $G_\gamma({\Cal A}^m) = \{0\}$ for each $\gamma\in \Gamma^{(m)}\setminus {\Cal B}^{(m)}$;
\item $G({\Cal A}^m) = G({\Cal A})^m$  as $R$-modules.
\item ${\Cal L}(f) =  {\bf M}(f)$ for each $f\in{\Cal A}^m$.
\end{itemize}

\section{Bilateral Gr\"obner bases}
 
Let 
${\Cal A} ={\Cal Q}/{\Cal I}$ be an  effectively given left $R$-module, 
endowed with its natural $\Gamma$-pseudovaluation ${\bf T}(\cdot)$
where the  semigroup $(\Gamma,\circ)$  satisfies 
\begin{itemize}
\item ${\Cal B}\subset\Gamma\subset\langle{\Bcc V}\rangle$ and 
\item the restriction of $<$ on 
$\Gamma$ is a semigroup ordering.
\end{itemize}
We denote ${\Cal G} = G({\Cal A})$,
by $\star$ the multiplication of ${\Cal A}$ and by $\ast$  the one of ${\Cal G}$.

For any set $F\subset{\Cal A}^m$ we denote, in function of $<$:

\begin{itemize}
\item ${\bf T}\{F\} := \{{\bf T}(f) : f\in F\}\subset {\Cal B}^{(m)};$
\item ${\bf M}\{F\} := \{{\bf M}(f) : f\in F\}\subset {\sf M}({\Cal A}^m).$
\item ${\bf T}_2(F) := {\Bbb I}_2({\bf T}\{F\}) = 
\{{\bf T}(\lambda\star f\star\rho)  : \lambda,\rho\in {\Cal B}, f\in F\}= 
\{\lambda\circ {\bf T}(f)\circ\rho  : \lambda,\rho\in {\Cal B}, f\in F\}\subset {\Cal B}^{(m)};$ 
\item ${\bf M}_2(F) := \{{\bf M}(a\lambda\star f\star b\rho)   : a\in R_\lambda\setminus\{0\}, b\in R_\rho\setminus\{0\},\lambda,\rho\in {\Cal B}, f\in F\}
= \{m\ast {\bf M}(f)\ast n : m,n\in {\sf M}({\Cal A}), f\in F\} 
\subset  {\sf M}({\Cal A}^m).$
\end{itemize}

\begin{Definition}\label{50cD1} 
Let ${\sf M} \subset {\cal A}^m$ be a bilateral ${\Cal A}$-module.
$F \subset {\sf M}$ will be called 

\begin{itemize}
\item a bilateral   {\em Gr\"obner
basis}
of  ${\sf M}$  if $F$ satisfies 
$${\bf M}\{{\sf M}\}= {\bf M}_2({\sf M})={\bf M}\{{\Bbb I}_2({\bf M}_2(F))\} ={\bf M}\{{\Bbb I}_2({\bf M}\{F\})\} ={\Bbb I}_2({\bf M}\{F\})\cap{\sf M}({\Cal A}^m),$$
{\em id est} if it satisfies the following condition:
\begin{itemize}
\item for each $f\in{\sf M},$ there are
$g_i\in F,$  
$\lambda_i,\rho_i\in {\cal B},
a_i\in R_{\lambda_i}\setminus\{0\}, b_i\in R_{\rho_i}\setminus\{0\}$ 
such that
\begin{itemize}
\renewcommand\labelitemiii{\bf --}
\item ${\bf T}(f) = \lambda_i\circ{\bf T}(g_i)\circ\rho_i$ for all $i$,
\item ${\bf M}(f)  = \sum_i a_i\lambda_i\ast {\bf M}(g_i)\ast  b_i\rho_i;$
\end{itemize}
\end{itemize}
\item a 
bilateral {\em strong Gr\"obner
basis} of ${\sf M}$ if 
it satisfies 
the following equivalent conditions:
\begin{itemize}
\item for each $f\in{\sf M}$ there is $g\in F$ such that
${\bf M}(g) \mid_2 {\bf M}(f),$ 
\item for each $f\in{\sf M}$ there are $g\in F,$  $a\in R_{\lambda}\setminus\{0\}, b\in R_{\rho}\setminus\{0\}, \lambda,\rho\in {\Cal B}$ such that
${\bf M}(f) =  a\lambda\ast {\bf M}(g)\ast   b\rho 
=  {\bf M}(a\lambda\star g\star  b\rho),$
\item ${\bf M}\{{\sf M}\} = {\bf M}_2({\sf M}) = {\bf M}_2(F)$.
\end{itemize}
\end{itemize}\end{Definition}

\begin{Definition} 
Let ${\sf M} \subset {\cal A}^m$ be a  bilateral ${\Cal A}$-module and
$F \subset {\sf M}$.
We say that $f\in{\cal A}^m\setminus\{0\}$ has 
\begin{itemize}
\item a bilateral (weak)  {\em Gr\"obner representation}
in terms of $F$ if it can be written as
$f = \sum_{i=1}^\mu a_i \lambda_i\star   g_i\star b_i \rho_i$, 
with $\lambda_i, \rho_i\in {\cal B}, a_i\in R_{\lambda_i}\setminus\{0\}, b_i\in R_{\rho_i}\setminus\{0\},g_i \in F,$ and
${\bf T}(f)\geq\lambda_i\circ{\bf T}(g_i)\circ \rho_i \Forall i;$
\item a bilateral {\em strong Gr\"obner representation} in terms
of $F$ if it can be written as 
$f = \sum_{i=1}^\mu a_i \lambda_i \star  g_i \star b_i\rho_i, $
with $\lambda_i, \rho_i\in {\cal B}, a_i\in R_{\lambda_i}\setminus\{0\}, b_i\in R_{\rho_i}\setminus\{0\}, g_i \in F,$ and
${\bf T}(f) =  \lambda_1\circ  {\bf T}(g_1)\circ\rho_1  > \lambda_i\circ  {\bf T}(g_i)\circ \rho_i \Forall i, 1 < i \leq \mu.$
\end{itemize}
For $f\in{\Cal A}^m\setminus\{0\}, F \subset{\Cal A}^m$, an element 
$g:=\NF(f, F)\in {\Cal A}^m$ is called a \index{form!normal}
\begin{itemize}
\item bilateral {\em (weak) normal form}  
of $f$ w.r.t. $F$, if
\begin{description}
\item[] $f - g\in{\Bbb I}_2(F)$ has a weak Gr\"obner representation wrt $F$ and 
\item[] $g \neq 0 \then {\bf M}(g) \notin{\bf M}\{{\Bbb I}_2({\bf M}\{F\})\};$
\end{description} 
\item bilateral  {\em strong normal form}  
of $f$ w.r.t. $F$, if
\begin{description}
\item[] $f - g\in{\Bbb I}_2(F)$ has a strong Gr\"obner representation  wrt $F$ and 
\item[] $g \neq 0 \then {\bf M}(g) \notin{\bf M}_2(F).$
\end{description}\end{itemize}
\end{Definition}

\begin{Remark}\label{50cRem6} 
As we noted above, ${\Cal G}:=G({\Cal A})$ and ${\Cal A}$, while coinciding as sets, do not necessarily coincide as rings nor as $R$-modules; thus
in general for $\lambda,\rho\in{\Cal B}, a\in R_{\lambda}\setminus\{0\}, b\in R_{\rho}\setminus\{0\}$ and $g\in{\Cal A}^m$,  $g={\bf M}(g)+p$, we don't have
$a\lambda\star{\bf M}(g)\star b\rho=a\lambda\ast{\bf M}(g)\ast b\rho$ but we could have
$$\tail(a\lambda\star{\bf M}(g)\star b\rho):=a\lambda\star{\bf M}(g)\star b\rho-a\lambda\ast{\bf M}(g)\ast b\rho \neq 0.$$

In such case, of course,
${\bf T}(\tail(a\lambda\star{\bf M}(g)\star b\rho)) < {\bf T}(a\lambda\star{\bf M}(g)\star b\rho);$
more exactly, either 
\begin{itemize}
\item $\lambda\circ{\bf T}(g)\circ\rho\in{\Cal B}^{(m)}$ in which case
$${\bf M}(a\lambda\star{\bf M}(g)\star b\rho) = a\lambda\ast{\bf M}(g)\ast b\rho$$ and
$a\lambda\star{\bf M}(g)\star b\rho = {\bf M}(a\lambda\star{\bf M}(g)\star b\rho)+\tail(a\lambda\star{\bf M}(g)\star b\rho);$
\item or $\lambda\circ{\bf T}(g)\circ\rho\in\Gamma^{(m)}\setminus{\Cal B}^{(m)}$ in which case
$$a\lambda\ast{\bf M}(g)\ast b\rho={\bf M}(a\lambda\star{\bf M}(g)\star b\rho) = 0 \And
a\lambda\star{\bf M}(g)\star b\rho =\tail(a\lambda\star{\bf M}(g)\star b\rho);$$ 
\end{itemize}
in both cases we have 
\begin{eqnarray*}
a\lambda\star g\star b\rho -a\lambda\ast{\bf M}(g)\ast b\rho 
&=&
a\lambda\star{\bf M}(g)\star b\rho -a\lambda\ast{\bf M}(g)\ast b\rho+a\lambda\star p\star b \rho
\\ &=&
\tail(a\lambda\star{\bf M}(g)\star b\rho) + a\lambda\star p \star b\rho=: h,
\end{eqnarray*}
with  
${\bf T}(h) < \lambda\circ {\bf T}(g)\circ\rho.$
\qed\end{Remark}

\begin{Lemma} Let $f\in{\Cal A}^m$; then for each
$g_i\in {\Cal A}^m$, 
$\lambda_i,\rho_i\in {\cal B},
a_i\in R_{\lambda_i}\setminus\{0\}, b_i\in R_{\rho_i}\setminus\{0\}$ 
which satisfy
\begin{itemize}
\item ${\bf T}(f) = \lambda_i\circ{\bf T}(g_i)\circ\rho_i,$ for each $i$,
\end{itemize}
the following are equivalent
\begin{enumerate}
\item ${\bf M}(f)  = \sum_i{\bf M}(a_i\lambda_i\star g_i \star  b_i\rho_i),$
\item ${\bf M}(f)  = \sum_i a_i\lambda_i\ast{\bf M}(g_i)\ast  b_i\rho_i,$
\item ${\bf T}\left(f -\sum_i a_i\lambda_i\star g_i \star  b_i\rho_i\right)<{\bf T}(f).$
\end{enumerate}\end{Lemma}

\begin{proof} Remark that the assumption
${\bf T}(f) = \lambda_i\circ{\bf T}(g_i)\circ\rho_i,$ for each $i$, grants, according Remark~\ref{50cRem6}, the equivalence $(1) \iff (2)$.

Moreover, denoting 
$q := f-{\bf M}(f)$, $p_i := g_i-{\bf M}(g_i)$,
$$h_i := a_i\lambda_i\star g_i \star  b_i\rho_i-a_i\lambda_i\ast{\bf M}(g_i)\ast  b_i\rho_i=
\tail(a_i\lambda_i\star{\bf M}(g_i)\star  b_i\rho_i)+a_i\lambda_i\star p_i \star  b_i\rho_i$$
and $h := q-\sum_i h_i$
 we have
\begin{eqnarray*}
f -\sum_i a_i\lambda_i\star g_i \star  b_i\rho_i 
&=&
{\bf M}(f) + q -  \sum_ia_i\lambda_i\ast{\bf M}(g_i)\ast  b_i\rho_i -
\sum_i h_i 
\\ &=&
{\bf M}(f) -  \sum_ia_i\lambda_i\ast{\bf M}(g_i)\ast  b_i\rho_i  +
 h.
\end{eqnarray*}

Thus, setting $\tau := {\bf T}(f) =  \lambda_i\circ{\bf T}(g_i)\circ\rho_i\in{\Cal B}^{(m)}$,
we have ${\bf T}(q)<\tau$ and ${\bf T}(h_i)<\tau$  for each $i$, so that 
 ${\bf T}(h)<\tau$.

Therefore 
${\bf M}(f) = \sum_ia_i\lambda_i\ast{\bf M}(g_i)\ast  b_i\rho_i =
 \sum_i{\bf M}(a_i\lambda_i\star g_i \star  b_i\rho_i)$ implies
 $$f -\sum_i a_i\lambda_i\star g_i \star  b_i\rho_i = h$$ so that
  ${\bf T}\left(f -\sum_i a_i\lambda_i\star g_i \star  b_i\rho_i\right) = {\bf T}(h) <{\bf T}(f)$
 proving $(2) \then (3)$.

Conversely,
$${\bf T}\left(f -\sum_i a_i\lambda_i\star g_i \star  b_i\rho_i\right) < {\bf T}(f)
\then
{\bf M}(f) - \sum_ia_i\lambda_i\ast{\bf M}(g_i)\ast  b_i\rho_i = 0.$$
\end{proof}

\begin{Theorem}\label{50cP1bil} 
For any set $F\subset {\Cal A}^m\setminus\{0\}$, among the
following
conditions:
\begin{enumerate}
\item $f\in{\Bbb I}_2(F) \iff $ it has a bilateral  strong Gr\"obner representation 
$$f = \sum_{i=1}^\mu a_i \lambda_i \star  g_i \star b_i\rho_i$$  
in terms of $F$ which further satisfies
$${\bf T}(f) =  \lambda_1\circ{\bf T}(g_1)\circ\rho_1 \And
\lambda_i\circ{\bf T}(g_i)\circ\rho_i> \lambda_{i+1}\circ{\bf T}(g_{i+1})\circ\rho_{i+1} \Forall i;$$
\item $f\in{\Bbb I}_2(F) \iff $ it has a bilateral  strong Gr\"obner representation in terms of $F$;
\item $F$ is a bilateral strong Gr\"obner basis of ${\Bbb I}_2(F)$; 
\item for each $f\in{\Cal A}^m\setminus\{0\}$ and any bilateral strong normal form $h$ of $f$ w.r.t. $F$ we have
$f\in{\Bbb I}_2(F) \iff h = 0;$ 
\item $f\in{\Bbb I}_2(F) \iff $ it has a bilateral  weak Gr\"obner representation in terms of $F$;
\item $F$ is a bilateral weak Gr\"obner basis of ${\Bbb I}_2(F)$;
\item for each $f\in{\Cal A}^m\setminus\{0\}$ and any bilateral weak normal form $h$ of $f$ w.r.t. $F$ we have
$f\in{\Bbb I}_2(F) \iff h = 0;$ 
\end{enumerate}
there are the implications
$$
\begin{array}{rcccccl}
(1)&\iff&(2)&\iff&(3)&\iff&(4) \\
&&\Downarrow&&\Downarrow&&\Downarrow\cr
&&(5)&\iff&(6)&\iff& (7)\cr\end{array}$$
If 
$R$ is a skew field
we have also the implication $(5) \then (2)$ and as a consequence the seven conditions are equivalent.
\end{Theorem}

\begin{proof} The implications  $(1) \then (2) \then (3)$, $(5) \then (6)$, $(2) \then (5)$, $(3) \then (6)$  and $(4) \then (7)$ are trivial. 

Ad $(3) \then (1)$:
for each $f\in{\Bbb I}_2(F)$ by assumption there are 
elements
$g\in F$,
$\lambda,\rho\in{\cal B},a\in R_\lambda\setminus\{0\}, b\in R_\rho\setminus\{0\}$, such that
$${\bf T}(f) = \lambda\circ{\bf T}(g)\circ\rho \And
{\bf M}(f) =a\lambda\ast {\bf M}(g)  \ast b\rho.$$

Thus ${\bf M}(a\lambda\star {\bf M}(g)  \star b\rho)= a\lambda\ast {\bf M}(g)  \ast b\rho = {\bf M}(f)$
and
denoting, for $f = {\bf M}(f)+q$ and $g = {\bf M}(g)+p$,
$$f_1 := f - a\lambda\star g  \star b\rho
=q-\tail(a\lambda\star{\bf M}(g)  \star b\rho) - a\lambda\star p  \star b\rho$$
we have
${\bf T}(f_1) < {\bf T}(f)$ so the claim follows 
by induction, since ${\Cal B}^{(m)}$ is well-ordered by $<$.

Ad $(6) \then (5)$: similarly, for each $f\in{\Bbb I}_2(F)$ by assumption there
are  elements
$g_i\in F$, 
$\lambda_i,\rho_i\in {\cal B},
a_i\in R_{\lambda_i}\setminus\{0\}, b_i\in R_{\rho_i}\setminus\{0\}$ 
such that
\begin{itemize}
\item ${\bf T}(f) = \lambda_i\circ{\bf T}(g_i)\circ\rho_i$ for all $i$,
\item ${\bf M}(f)  = \sum_i a_i\lambda_i\ast {\bf M}(g_i)\ast  b_i\rho_i.$
\end{itemize}

Thus ${\bf T}(f -\sum_i a_i\lambda_i\star g_i \star  b_i\rho_i)<{\bf T}(f)$ and 
it is then sufficient to denote 
$f_1 := f -\sum_i a_i\lambda_i\star g_i \star  b_i\rho_i$ 
in order to deduce the claim 
by induction.

Ad $(3) \then (4)$ and $(6) \then (7)$: either
\begin{itemize}
\item $h=0$ and $f=f-h\in{\Bbb I}_2(F)$ or
\item $h\neq 0$, ${\bf M}(h) \notin {\bf M}({\Bbb I}_2(F))$, $h\notin{\Bbb I}_2(F)$ and $f\notin{\Bbb I}_2(F)$.
\end{itemize}

Ad  $(4) \then (2)$ and $(7) \then (5)$: for each $f\in{\Bbb I}_2(F)$, its normal form is  $h=0$ and
$f = f-h$ has a strong (resp.: weak) Gr\"obner representation in
terms of $F$.

Ad $(5) \then (2)$: let $f\in{\Bbb I}_2(F)\setminus\{0\};$ since $R$ is a skew field, (5) implies the existence of elements 
$g\in F$, 
$\lambda,\rho\in {\cal B},$
such that
${\bf T}(f) = \lambda\circ{\bf T}(g)\circ\rho =: \tau;$
thus denoting $d\in R\setminus\{0\}$ the value 
which satisfies
$$d\tau = {\bf M}(\lambda\star g \star \rho) = \lambda\ast {\bf M}(g)\ast\rho,$$
we have
$${\bf M}(f) =   
\lc(f)d^{-1}d \tau  = 
\lc(f)d^{-1}\lambda\ast{\bf M}(g)\ast\rho=
{\bf M}\left((\lc(f)d^{-1}\lambda)\star g\star \rho\right)$$
as required.
\end{proof}

\section {Weispfenning multiplication}

In proposing a Buchberger Theory for a class of Ore-like rings, {\em id est} Weispfenning rings \cite{W}, \cite[IV.49.11,IV.50.13.6]{SPES}
${\Bbb Q}\langle x,Y\rangle/{\Bbb I}(Yx-x^eY), e\in{\Bbb N}, e>1$, Weispfenning considered, given a basis $F$, the restricted
module $${\Bbb I}_W(F) := \Span_{\Bbb Q}\{x^a fY^b, (a,b)\in{\Bbb N}^2\}$$ and computed a restricted Gr\"obner bases $G$ which grants to each element $f\in{\Bbb I}_W(F)$ a restricted Gr\"obner representation
$$f=\sum_{i=1}^\mu c_ix^{a_i}g_iY^{b_i} : \deg_Y(f)\geq\deg_Y(g_i)+b_i, c_i\in {\Bbb Q},  (a_i,b_i)\in{\Bbb N}^2, g_i\in G,$$ 
to be extended, in a second step, to the required basis by an adaptation of Kandri-Rody--- Weispfenning completion \cite{KrW}\cite[IV.49.5.2]{SPES}.

We can interpret this construction as a multiplication on the monomial set 
$${\sf M}({\Cal A}):=\{ct : t\in{\Cal B},c\in R_t\setminus\{0\}\}$$ which becomes, by distribution, a multiplication in  ${\Cal A}$.

\begin{Definition} Setting, for each $m_1=a_1\tau_1,m_2=a_2\tau_2\in{\sf M}({\Cal A})$
$$m_1\diamond m_2 := \left(a_1a_2\right)\left(\tau_2\circ\tau_1\right)$$
{\em Weispfenning multiplication} is the associative multiplication
$$\diamond : {\Cal A}\times{\Cal A}\to{\Cal A}$$ defined as $$f\diamond g=\sum_{\tau\in\supp(f)}\sum_{\omega\in\supp(g)} m_\tau\diamond n_\omega =
\sum_{\tau\in\supp(f)}\sum_{\omega\in\supp(g)} c(f,\tau)c(g,\omega)\omega\tau$$
for each $f = \sum_{\tau\in\supp(f)} m_\tau,  m_\tau = c(f,\tau)\tau$ and $g = \sum_{\omega\in\supp(g)} n_\omega,  n_\omega = c(g,\omega)\omega$.
\end{Definition}

Note that $\diamond$ is commutative when ${\Cal A}$ is a twisted monoid ring $R[{\sf S}]$ over a commutative ring $R$ and a commutative monoid ${\sf S}$, as polynomial rings, solvable polynomial rings \cite{KrW,Kr},\cite[IV.49.5]{SPES}, multivariate Ore extensions \cite{P1,P2,BGV,labOrE} \ldots.

The intuition of Weispfenning can be formulated by remarking that its effect is to transform a bilateral problem into a left one.
Thus the construction proposed in \cite{W} simply reformulates the one stated in \cite{KrW}; in an analogous way the reformulation of the (commutaive) Gebauer--M\"oller criteria \cite{GM} for detecting useless S-pairs was easily performed in \cite{labOrE} in the context of multivariate Ore extensions by means of Weispfenning multiplication.

Our aim is therefore to apply $\diamond$ to reduce the computation of Gebauer--M\"oller sets for the bilateral case to the trivial right case where efficient solutions are already available \cite{M},\cite[IV.47.2.3]{SPES}.

We note that  Weispfenning construction is a smoother special case of the construction proposed by Pritchard \cite{Pr1,Pr2},\cite[IV.47.5]{SPES} for reformulating bilateral modules in ${\Bbb D}\langle {\Bcc X}\rangle$ as left modules in a monoid ring
${\Bbb D}[\langle{\Bcc X}\rangle^\star]$ where the monoid $\langle{\Bcc X}\rangle^\star$ is properly defined in terms of
$\langle{\Bcc X}\rangle$.

\section{Restricted Gr\"obner bases}

Following Weispfenning's intuition \cite{W} we further denote
\begin{itemize}
\item  ${\Bbb I}_W(F)\subset{\Cal A}^m$ the {\em restricted}
module generated by $F$,
\begin{eqnarray*}
{\Bbb I}_W(F) &:=& \Span_{R}(af\star \rho   : a\in
R\setminus\{0\}, \rho\in {\Cal B}, f\in F),\\
&=& \Span_{R}(m\diamond f : m\in{\sf M}({\Cal A}^m), f\in F\},
\end{eqnarray*}
\item ${\bf T}_W(F) := {\Bbb I}_R({\bf T}\{F\}) =
\{{\bf T}(f\star\rho)  : \rho\in {\Cal B}, f\in F\}
= 
\{{\bf T}(f)\circ\rho  : \rho\in {\Cal B}, f\in F\}\subset {\Cal B}^{(m)};$ 
\item ${\bf M}_W(F) := \{{\bf M}(a f\star \rho)  : a\in
R\setminus\{0\},\rho\in {\Cal B}, f\in F\}
= \{a {\bf M}(f)\ast \rho : a\in
R\setminus\{0\},\rho\in {\Cal B}, f\in F\}
= \{m\diamond {\bf M}(f) : m\in{\sf M}({\Cal A}^m), f\in F\}\subset  {\sf M}({\Cal A}^m).$
\end{itemize}

\begin{Definition}  
Let ${\sf M} \subset {\cal A}^m$ be a retricted ${\Cal A}$-module.
$F \subset {\sf M}$ will be called 

\begin{itemize}
\item a restricted  {\em Gr\"obner
basis}
of  ${\sf M}$  if $F$ satisfies 
$${\bf M}\{{\sf M}\}= {\bf M}_W({\sf M}) ={\bf M}\{{\Bbb I}_W({\bf M}_W(F))\}={\bf M}\{{\Bbb I}_W({\bf M}\{F\})\} ={\Bbb I}_W({\bf M}\{F\})\cap{\sf M}({\Cal A}^m),$$
{\em id est} if it satisfies the following condition:
\begin{itemize}
\item for each $f\in{\sf M},$ there are
$g_i\in F$, 
$\rho_i\in {\cal B},
a_i\in R\setminus\{0\}$ 
such that
\begin{itemize}
\renewcommand\labelitemiii{\bf --}
\item ${\bf T}(f) = {\bf T}(g_i)\circ\rho_i$ for all $i$,
\item ${\bf M}(f)  = \sum_i a_i{\bf M}(g_i)\ast\rho_i = \sum_i a_i\rho_i\diamond {\bf M}(g_i);$
\end{itemize}
\end{itemize}
\item a 
restricted {\em strong Gr\"obner
basis} of ${\sf M}$ if 
it satisfies 
the following equivalent conditions:
\begin{itemize}
\item for each $f\in{\sf M}$ there is $g\in F$ such that
${\bf M}(g) \mid_W {\bf M}(f),$ 
\item for each $f\in{\sf M}$ there are $g\in F$, $a\in
R\setminus\{0\}$, $\ \rho\in {\Cal B}$ such that
$${\bf M}(f) =  a{\bf M}(g)\ast \rho 
=  {\bf M}(ag\star \rho)=  {\bf M}(a\rho\diamond g),$$
\item ${\bf M}\{{\sf M}\} = {\bf M}_W({\sf M}) = {\bf M}_W(F)$.
\end{itemize}
\end{itemize}\end{Definition}

\begin{Definition} 
Let ${\sf M} \subset {\cal A}^m$ be a  restricted ${\Cal A}$-module and
$F \subset {\sf M}$.
We say that $f\in{\cal A}^m\setminus\{0\}$ has 
\begin{itemize}
\item a restricted (weak)  {\em Gr\"obner representation}
in terms of $F$ if it can be written as
$f = \sum_{i=1}^\mu a_i   g_i\star  \rho_i = \sum_{i=1}^\mu a_i\rho_i\diamond g_i$, 
with $\rho_i\in {\cal B}, a_i\in R\setminus\{0\},g_i \in F,$ and
${\bf T}(f)\geq{\bf T}(g_i)\circ \rho_i \Forall i;$
\item a restricted {\em strong Gr\"obner representation} in terms
of $F$ if it can be written as 
$f = \sum_{i=1}^\mu a_i g_i\star  \rho_i = \sum_{i=1}^\mu a_i\rho_i\diamond g_i$, 
with $\rho_i\in {\cal B}, a_i\in R\setminus\{0\},g_i \in F,$ and
${\bf T}(f) =  {\bf T}(g_1)\circ\rho_1  > {\bf T}(g_i)\circ \rho_i \Forall i, 1 < i \leq \mu.$
\end{itemize}
For $f\in{\Cal A}^m\setminus\{0\}, F \subset{\Cal A}^m$, an element 
$g:=\NF(f, F)\in {\Cal A}^m$ is called a 
\begin{itemize}
\item restricted {\em (weak) normal form}  
of $f$ w.r.t. $F$, if
\begin{description}
\item[] $f - g\in{\Bbb I}_W(F)$ has a restricted weak Gr\"obner representation wrt $F$, and 
\item[] $g \neq 0 \then {\bf M}(g) \notin{\bf M}\{{\Bbb I}_W({\bf M}\{F\})\};$
\end{description} 
\item restricted {\em strong normal form}  
of $f$ w.r.t. $F$, if
\begin{description}
\item[] $f - g\in{\Bbb I}_W(F)$ has a restricted strong Gr\"obner representation wrt $F$, and 
\item[] $g \neq 0 \then {\bf M}(g) \notin{\bf M}_W(F).$
\end{description}\end{itemize}
\end{Definition}

\begin{Lemma} Let $f\in{\Cal A}^m$; then for each
$g_i\in {\Cal A}^m$, 
$\rho_i\in {\cal B},
a_i\in R\setminus\{0\}$ 
which satisfy
\begin{itemize}
\item ${\bf T}(f) = {\bf T}(g_i)\circ\rho_i,$ for each $i$,
\end{itemize}
the following are equivalent
\begin{enumerate}
\item ${\bf M}(f)  
= \sum_i{\bf M}(a_i \rho_i\diamond g_i),$
\item ${\bf M}(f)  = \sum_i a_i {\bf M}(g_i)\ast  \rho_i ,$
\item ${\bf T}\left(f -\sum_i a_i \rho_i\diamond g_i\right)<{\bf T}(f).$
\end{enumerate}\end{Lemma}

\begin{proof} Remark that the assumption
${\bf T}(f) =  {\bf T}(g_i)\circ\rho_i,$ for each $i$, grants, according Remark~\ref{50cRem6}, the equivalence $(1) \iff (2)$.

Moreover, denoting, for each $\rho\in{\Cal B}, a\in R\setminus\{0\}$ and $g\in{\Cal A}^m$
$$\tail(a\rho\diamond {\bf M}(g) ):=a\rho\diamond {\bf M}(g)-a{\bf M}(g)\ast\rho$$
and setting
$q := f-{\bf M}(f)$, $p_i := g_i-{\bf M}(g_i)$,
$$h_i := a_i  \rho_i\diamond g_i -a_i {\bf M}(g_i)\ast \rho_i=
\tail( a_i  \rho_i\diamond {\bf M}(g_i))+a_i  \rho_i\diamond p_i$$
and $h := q-\sum_i h_i$
 we have
\begin{eqnarray*}
f -\sum_i a_i  \rho_i\diamond g_i  
&=&
{\bf M}(f) + q -  \sum_i a_i {\bf M}(g_i)\ast  \rho_i -
\sum_i h_i 
\\ &=&
{\bf M}(f) -  \sum_ia_i {\bf M}(g_i)\ast  \rho_i  +
 h.
\end{eqnarray*}

Thus, setting $\tau := {\bf T}(f) = {\bf T}(g_i)\circ\rho_i\in{\Cal B}^{(m)}$,
we have ${\bf T}(q)<\tau$ and ${\bf T}(h_i)<\tau$  for each $i$, so that 
 ${\bf T}(h)<\tau$.

Therefore 
${\bf M}(f) = \sum_ia_i {\bf M}(g_i)\ast \rho_i =
 \sum_i{\bf M}(a_i  \rho_i\diamond g_i)$ implies
 $f -\sum_i  a_i  \rho_i\diamond g_i  = h$ so that
  ${\bf T}\left(f -\sum_i a_i  \rho_i\diamond g_i\right) = {\bf T}(h) <{\bf T}(f)$
 proving $(2) \then (3)$.

Conversely,
$${\bf T}\left(f -\sum_ia_i  \rho_i\diamond  g_i \right) < {\bf T}(f)
\then
{\bf M}(f) - \sum_ia_i {\bf M}(g_i)\ast \rho_i = 0.$$
\end{proof}

\begin{Theorem}\label{50cP1res}
For any set $F\subset {\Cal A}^m\setminus\{0\}$, among the
following
conditions:
\begin{enumerate}
\item $f\in{\Bbb I}_W(F) \iff $ it has a restricted strong Gr\"obner representation 
$$f= \sum_{i=1}^\mu a_i  g_i \star \rho_i = \sum_{i=1}^\mu a_i\rho_i\diamond  g_i$$  
in terms of $F$ which further satisfies
$${\bf T}(f) =   {\bf T}(g_1)\circ\rho_1  > \cdots >  {\bf T}(g_i)\circ\rho_i>  {\bf T}(g_{i+1})\circ\rho_{i+1};$$
\item $f\in{\Bbb I}_W(F) \iff $ it has a restricted strong Gr\"obner representation in terms of $F$;
\item $F$ is a restricted strong Gr\"obner basis of ${\Bbb I}_W(F)$; 
\item for each $f\in{\Cal A}^m\setminus\{0\}$ and any restricted strong 
normal form $h$ of $f$ w.r.t. $F$ we have
$f\in{\Bbb I}_W(F) \iff h = 0;$ 
\item $f\in{\Bbb I}_W(F) \iff $ it has a  restricted weak Gr\"obner representation in terms of $F$;
\item $F$ is a  restricted weak Gr\"obner basis of ${\Bbb I}_W(F)$;
\item for each $f\in{\Cal A}^m\setminus\{0\}$ and any restricted weak normal form $h$ of $f$ w.r.t. $F$ we have
$f\in{\Bbb I}_W(F) \iff h = 0.$ 
\end{enumerate}
there are the implications
$$
\begin{array}{rcccccl}
(1)&\iff&(2)&\iff&(3)&\iff&(4) \\
&&\Downarrow&&\Downarrow&&\Downarrow\cr
&&(5)&\iff&(6)&\iff& (7)\cr\end{array}$$
If 
$R$ is a skew field
we have also the implication $(5) \then (2)$ and as a consequence the seven conditions are equivalent. 
\end{Theorem}

 \begin{proof}
  The implications  $(1) \then (2) \then (3)$, $(5) \then (6)$, $(2) \then (5)$, $(3) \then (6)$  and $(4) \then (7)$ are trivial.

Ad $(3) \then (1)$:
for each $f\in{\Bbb I}_W(F)$ by assumption there are 
elements
$g\in F$,
$\rho\in{\cal B},a\in R\setminus\{0\}$, such that
$${\bf T}(f) =  {\bf T}(g)\circ\rho \And
{\bf M}(f) =a {\bf M}(g)  \ast \rho.$$

Thus ${\bf M}(a  \rho\diamond{\bf M}(g))= a {\bf M}(g)  \ast \rho = {\bf M}(f)$
and
denoting, for $f = {\bf M}(f)+q$ and $g = {\bf M}(g)+p$,
$$f_1 := f - a \rho\diamond  g 
=q-\tail( a \rho\diamond {\bf M}(g)) -  a \rho\diamond p$$
we have
${\bf T}(f_1) < {\bf T}(f)$ so the claim follows 
by induction, since ${\Cal B}^{(m)}$ is well-ordered by $<$.

\medskip

Ad $(6) \then (5)$: similarly, for each $f\in{\Bbb I}_W(F)$ by assumption there
are  elements
$g_i\in F$, 
$\rho_i\in {\cal B},
a_i\in R\setminus\{0\}$ 
such that
\begin{itemize}
\item ${\bf T}(f) =  {\bf T}(g_i)\circ\rho_i$ for all $i$,
\item ${\bf M}(f)  = \sum_i a_i  {\bf M}(g_i)\ast \rho_i= \sum_i a_i\rho_i \diamond {\bf M}(g_i).$
\end{itemize}

Thus ${\bf T}(f -\sum_i a_i  \rho_i\diamond g_i)<{\bf T}(f)$ and 
it is then sufficient to denote 
$f_1 := f -\sum_i a_i  \rho_i\diamond g_i$ 
in order to deduce the claim 
by induction.

\medskip 

Ad $(3) \then (4)$ and $(6) \then (7)$: either
\begin{itemize}
\item $h=0$ and $f=f-h\in{\Bbb I}_W(F)$ or
\item $h\neq 0$, ${\bf M}(h) \notin {\bf M}_W({\Bbb I}_W(F))$, $h\notin{\Bbb I}_W(F)$ and $f\notin{\Bbb I}_W(F)$.
\end{itemize}

\medskip

Ad  $(4) \then (2)$ and $(7) \then (5)$: for each $f\in{\Bbb I}_W(F)$, its normal form is  $h=0$ and
$f = f-h$ has a strong (resp.: weak) Gr\"obner representation in
terms of $F$.

\medskip

Ad $(5) \then (2)$: let $f\in{\Bbb I}_W(F)\setminus\{0\};$ since $R$ is a skew field, (5) implies the existence of elements 
$g\in F$, 
$\rho\in {\cal B},$
 such that
${\bf T}(f) =  {\bf T}(g)\circ\rho =: \tau;$
thus denoting $d\in R\setminus\{0\}$ the value 
which satisfies
$$d\tau = {\bf M}( \rho\diamond g) =  {\bf M}(g)\ast\rho,$$
we have
$${\bf M}(f) =  
\lc(f)d^{-1}d \tau  = 
\lc(f)d^{-1} {\bf M}(g)\ast\rho=
{\bf M}\left((\lc(f)d^{-1}) \rho\diamond g\right)$$
as required.  
 \end{proof}

\section{Lifting Theorem for Restricted Modules}

Given the finite set
$$F := \{g_1,\ldots,g_u\}\subset {\Cal A}^m, 
g_i = {\bf M}(g_i)-p_i =: a_i \tau_i {\bf e}_{\iota_i} - p_i,$$
let us now  denote ${\sf M}$ 
the restricted module ${\sf M} := {\Bbb I}_W(F)$ endowed with its natural $\Gamma$-pseudovaluation ${\bf T}(\cdot)$.

  Considering both the left $R$-module
$R\otimes_{R}{\Cal A}^{\op}$ and the  left $R$-module
$R\otimes_{R}{\Cal G}^{\op}$, which, as sets, coincide, 
we impose on the left $R$-module $\left(R\otimes_{R}{\Cal A}^{\op}\right)^u$, whose 
canonical basis is denoted $\{e_1,\ldots,e_u\}$ and whose
 generic element
 has the shape
$$\sum_i a_i  e_{l_i}    \rho_i, 
 \rho_i\in{\Cal B}, a_i\in R\setminus\{0\}, 1\leq l_i \leq u,$$
the $\Gamma^{(m)}$-pseudovaluation -- compatible with the natural $\Gamma$-pseudovaluation of ${\Cal A}$ --
$$w :\left(R\otimes_{R}{\Cal A}^{\op}\right)^u \to \Gamma^{(m)}$$
defined for each $\sigma  := \sum_i a_i   e_{l_i}    \rho_i\in\left(R\otimes_{R}{\Cal A}^{\op}\right)^u\setminus\{0\}$
 as
$$w(\sigma) := 
 \max_<\{ {\bf T}(g_{l_i})\circ\rho_i,\, \rho_i \in {\Cal B}\} \in \Gamma^{(m)}$$
 so that 
$G\left(\left(R\otimes_{R}{\Cal A}^{\op}\right)^u\right) =
\left(G\left(R\otimes_{R}{\Cal A}^{\op}\right)\right)^u=
\left(R\otimes_{R}{\Cal G}^{\op}\right)^u$
and its corresponding $\Gamma^{(m)}$-homogeneous -- of $\Gamma^{(m)}$-degree
$w(\sigma)$ -- {\em leading form} 
is
$${\Cal L}(\sigma) := \sum_{h\in H} a_h  e_{l_h} \rho_h \in \left(R\otimes_{R}{\Cal G}^{\op}\right)^u 
\Where
H := \{ h : \tau_{l_h}\circ\rho_h{\bf e}_{\iota_{l_h}} = w(\sigma) \}.$$ 

We can therefore consider the morphisms
\begin{eqnarray*}
{\Frak s}_W :  \left(R\otimes_{R}{\Cal G}^{\op}\right)^u \to {\Cal G}^m
&:& {\Frak s}_W\left(\sum_i a_i  e_{l_i}  \rho_i\right) :=  
\sum_i a_i  {\bf M}(g_{l_i})\ast \rho_i,
\\
{\Frak S}_W :  \left(R\otimes_{R}{\Cal A}^{\op}\right)^u \to {\Cal A}^m
&:& {\Frak S}_W\left(\sum_i a_i e_{l_i}  \rho_i\right) := 
\sum_i a_i  g_{l_i}  \star \rho_i.
\end{eqnarray*}

We can equivalently reformulate this setting in terms of Weispfenning multiplication considering 
the morphisms
\begin{eqnarray*}
{\Frak s}_W : {\Cal G}^u \to {\Cal G}^m 
&:& {\Frak s}_W\left(\sum_{i=1}^u \left(\sum_{\rho\in{\Cal B}} a_{i\rho} \rho\right) e_i\right) := \sum_{i=1}^u \sum_{\rho\in{\Cal B}} a_{i\rho}   {\bf M}(g_i)\ast\rho,
\\
{\Frak S}_W : {\Cal A}^u \to {\Cal A}^m 
&:& {\Frak S}_W\left(\sum_{i=1}^u \left(\sum_{\rho\in{\Cal B}} a_{i\rho} \rho\right) e_i\right) := \sum_{i=1}^u \sum_{\rho\in{\Cal B}}  a_{i\rho} \rho\diamond g_i,
\end{eqnarray*}
where  the symbols $\{e_1,\ldots,e_u\}$ denote the common canonical basis of 
 ${\Cal A}^u$ and ${\Cal G}^u$, which,  as  sets,  coincide and which satisfy 
 ${\Cal G}^u = G({\Cal A})^u = G({\Cal A}^u)$ under the pseudovaluation 
$w : {\Cal A}^u \to \Gamma^{(m)}$  defined,
for each 
$$\sigma  := \sum_{i=1}^u \left(\sum_{\rho\in{\Cal B}} a_{i\rho} \rho\right) e_i \in{\Cal A}^u\setminus\{0\}$$
by
$$w(\sigma) := \max_<\left\{{\bf T}(g_i)\circ\rho : a_{i\rho}\neq 0\right\}\in\Gamma^{(m)}.$$

The corresponding $\Gamma^{(m)}$-homogeneous -- of $\Gamma^{(m)}$-degree
$w(\sigma)$ --
{\em leading form} is
$${\Cal L}(\sigma) := 
\sum_{i=1}^u \left(\sum_{\rho\in B_i} a_{i\rho}\rho\right) e_i\in{\Cal G}^u$$
where, for each $i$ we set
$B_i := \left\{\rho\in{\Cal B} :{\bf T}(g_i)\circ\rho = w(\sigma)\right\}.$
  
\begin{Definition} \

\begin{itemize}

\item if $u\in \ker({\Frak s}_W)$ is $\Gamma^{(m)}$-homogeneous
and $U\in \ker({\Frak S}_W)$ is such that $u = {\Cal L}(U)$, we say that 
$u$ {\em lifts} to $U$, or $U$ is a {\em lifting} of $u$, or simply $u$ {\em has
a lifting};\index{lift}
\item a restricted {\em Gebauer--M\"oller set}\index{Gebauer--M\"oller!set} for $F$ is any 
$\Gamma^{(m)}$-homogeneous basis of
$\ker({\Frak s}_W)$;
\item for each $\Gamma^{(m)}$-homogeneous element 
$\sigma=\sum_i a_i  e_{l_i}    \rho_i\in \left(R\otimes_{R}{\Cal A}^{\op}\right)^u$
-- or, equivalently,
$$\sigma=\sum_{i=1}^u a_i\rho_ie_i\in{\Cal A}^u\setminus\{0\}, a_i\neq 0, \then{\bf T}(g_i)\circ\rho_i=w(\sigma),$$
we say that ${\Frak S}_W(\sigma)$  has a restricted
 {\em quasi-Gr\"obner representation}\index{Gr\"obner!representation!quasi}
 in terms of $F$ if it can be written as
$${\Frak S}_W(\sigma)= 
\sum_{l=1}^\mu a_l g_{l}\star  \rho_l = 
\sum_{l=1}^\mu a_l\rho_l\diamond g_{l} 
:
 \rho_l\in{\Cal B}, a_l\in R\setminus\{0\}, g_l \in F$$
with 
$w(\sigma)> {\bf T}(a_l   g_{l}\star \rho_l) =  {\bf T}(g_{l})\circ \rho_l$ for each $l,$
-- or, equivalently,
$${\Frak S}_W(\sigma)= \sum_{i=1}^u h_i\diamond g_i, h_i\in{\Cal A}^u, w(\sigma)> {\bf T}(g_{i})\circ{\bf T}(h_{i}).$$
\item Denoting for each set  
$S\subset {\sf M}$,
${\Cal L}\{S\} := \{{\Cal L}(g) : g\in S\}\subset G({\sf M})$,
 a set $B\subset  {\sf M}$ is called a restricted {\em standard basis}\index{standard!basis} 
of {\sf M} if  
$${\Bbb I}_W({\Cal L}\{B\}) = {\Bbb I}_W({\Cal L}\{{\sf M}\}).$$
\qed\end{itemize}\end{Definition}

 
 \begin{Theorem}[M\"oller--Pritchard]\cite{M,Pr1,Pr2}\label{50BiLiTh} 
With the present notation and denoting 
$\GM_W(F)$   any  restricted Gebauer--M\"oller set for $F$, the following
conditions are equivalent: 
\begin{enumerate}
\item $F$ is a restricted Gr\"obner basis of ${\sf M}$;
\item $f \in {\sf M}\iff f$  has a restricted  Gr\"obner representation in terms of $F$;
\item for each $\sigma\in\GM_W(F)$, 
the restricted S-polynomial ${\Frak S}_W(\sigma)$ has a restricted quasi-Gr\"obner representation
${\Frak S}_W(\sigma)= \sum_{l=1}^\mu a_l\rho_l\diamond  g_{l} = \sum_{l=1}^\mu a_l  g_{l}\star  \rho_l, $ 
in terms of $F$;
\item each $\sigma\in\GM_W(F)$ has a lifting $\lift(\sigma)$;
\item each $\Gamma^{(m)}$-homogeneous element 
$u\in\ker({\Frak s}_W)$ has a lifting $\lift(u)$.
\end{enumerate}
\end{Theorem}

 \begin{proof}
  \
\begin{description}
\item[$(1) \then (2)$] is Theorem~\ref{50cP1res} $(6) \then (5)$.
\item[$(2) \then (3)$] ${\Frak S}_W(\sigma)\in{\sf M}$ and ${\bf T}({\Frak S}_W(\sigma)) < w(\sigma)$.
\item[$(3) \then (4)$] Let 
$${\Frak S}_W(\sigma)= 
\sum_{i=1}^\mu a_i\rho_i\diamond   g_{l_i} = 
\sum_{i=1}^\mu a_i  g_{l_i}\star    \rho_i,
w(\sigma)> \tau_{l_i}\circ \rho_i {\bf e}_{\iota_{l_i}}$$
be a restricted 
quasi-Gr\"obner representation in terms of $F$; then $$\lift(\sigma) := \sigma -
\sum_{i=1}^\mu a_i  e_{l_i} \rho_i$$ is the required lifting of $\sigma$.  
\item[$(4) \then (5)$] Let 
$$u := \sum_i a_i  e_{l_i}   \rho_i\in\left(R\otimes_{R}{\Cal G}^{\op}\right)^u,
  \tau_{l_i}\circ\rho_i  {\bf e}_{\iota_{l_i}}= w(u),$$
be a $\Gamma^{(m)}$-homogeneous element in $\ker({\Frak s}_W)$ of $\Gamma^{(m)}$-degree $w(u)$.

\medskip

Then there are $\rho_\sigma\in{\Cal B},a_\sigma\in R \setminus\{0\},$
for which
$$u= \sum_{\sigma\in\GM_W(F)} 
a_\sigma  \sigma\ast  \rho_\sigma  
, 
w(\sigma)\circ\rho_\sigma = w(u).$$

\medskip 
  
For each $\sigma\in\GM_W(F)$ denote
{\small $$\bar{\sigma} := \sigma - \lift(\sigma) =  {\Cal L}(\lift(\sigma)) -
\lift(\sigma) := 
\sum_{i=1}^{\mu_\sigma} a_{i\sigma}  e_{l_{i\sigma}}   \rho_{i\sigma}
\in\left(R\otimes_{R}{\Cal A}^{\op}\right)^u$$}
and remark that
 $ \tau_{l_i} \circ\rho_{i\sigma}{\bf e}_{\iota_{l_i}}\leq w(\bar{\sigma}) < w(\sigma)$ and 
 ${\Frak S}_W(\bar{\sigma}) = {\Frak S}_W(\sigma).$

\medskip

It is sufficient to define
 $$\lift(u) := \sum_{\sigma\in\GM_W(F)} a_\sigma  \lift(\sigma)\star \rho_\sigma = \sum_{\sigma\in\GM_W(F)} a_\sigma \rho_\sigma\diamond \lift(\sigma)$$
and 
$$\bar{u}:= \sum_{\sigma\in\GM_W(F)} a_\sigma  \bar{\sigma}\star \rho_\sigma  = \sum_{\sigma\in\GM_W(F)} a_\sigma \rho_\sigma\diamond \bar{\sigma}$$
to obtain
$$\lift(u) = u - \bar{u},{\Cal L}(\lift(u)) = u,{\Frak S}_W(\bar{u}) = {\Frak S}_W(u), 
{\Frak S}_W(\lift(u)) = 0.$$

\item[$(5) \then (1)$] Let
$g\in{\sf M}$, so that there are 
$ \rho_i\in{\Cal B}, a_i\in R\setminus\{0\}, 1\leq l_i \leq u,$ 
such that
$\sigma_1 := \sum_{i=1}^\mu a_i  e_{l_i}  \rho_i\in  \left(R\otimes_{R}{\Cal A}^{\op}\right)^u$
satisfies 
$$g = {\Frak S}_W(\sigma_1)=  \sum_{i=1}^\mu a_i g_{l_i}\star  \rho_i =  \sum_{i=1}^\mu a_i\rho_i\diamond g_{l_i}.$$
Denoting $H := \{i :  {\bf T}(g_{l_i})\circ \rho_i = 
 \tau_{l_i}\circ \rho_i{\bf e}_{\iota_{l_i}} = 
w(\sigma_1) \}$, then 
either 

\begin{itemize}
\item $w(\sigma_1) = {\bf T}(g)\in{\Cal B}^{(m)}$ so that, for each $i\in H$,
$ {\bf M}(a_i  {\bf M}(g_{l_i})\star \rho_i) =  a_i  {\bf M}(g_{l_i})\ast  \rho_i$
and
$${\bf M}(g) = \sum_{i\in H}  a_i {\bf M}(g_{l_i})\ast  \rho_i\in
{\bf M}\{{\Bbb I}_W({\bf M}\{F\})\},$$ and we are through, or
\item ${\bf T}(g) < w(\sigma_1)$, in which case\footnote{Compare Remark~\ref{50cRem6}.} 
$0 = \sum_{i\in H}  a_i {\bf M}(g_{l_i})\ast  \rho_i = {\Frak s}_W({\Cal L}(\sigma_1))$
and 
the $\Gamma^{(m)}$-homoge\-neous element 
${\Cal L}(\sigma_1)\in \ker({\Frak s}_W)$
has a  lifting $$U := {\Cal L}(\sigma_1)-
\sum_{j=1}^\nu a_j e_{l_j} \rho_j\in \left(R\otimes_{R}{\Cal A}^{\op}\right)^u$$
with $$
\sum_{j=1}^\nu a_j\rho_j\diamond g_{l_j} = 
\sum_{i\in H} a_i\rho_i\diamond  g_{l_i}\And
\tau_{l_j}\circ \rho_j{\bf e}_{\iota_{l_j}} < w(\sigma_1)$$
so that
$g = {\Frak S}_W(\sigma_2)$ and  $w(\sigma_2)< w(\sigma_1)$
for
$$\sigma_2 := 
\sum_{i\not\in H} a_i  e_{l_i}  \rho_i + 
\sum_{j=1}^\nu a_j   e_{l_j}   \rho_j
\in \left(R\otimes_{R}{\Cal A}^{\op}\right)^u$$
\end{itemize}

and the claim follows by the well-orderedness of $<$.

\end{description} \end{proof}

\begin{Theorem}[Janet---Schreier]\cite{Jan1,Sch1,Sch2} \

With the same notation the equivalent conditions (1-5)  imply that 
\begin{enumerate}\setcounter{enumi}{5}
\item $\{\lift(\sigma) : \sigma\in{\GM}_W(F)\}$ is a restricted standard basis of $\ker({\Frak S}_W)$. 
\end{enumerate}
\end{Theorem}
\begin{proof} Let 
$\sigma_1 := 
 \sum_{i=1}^\mu a_i   e_{l_i} \rho_i\in\ker({\Frak S}_W)\subset \left(R \otimes_{R}{\Cal A}^{\op}\right)^u.$ 

Denoting
 $H := \{i :  \tau_{l_i}\circ \rho_i{\bf e}_{\iota_{l_i}} = 
w(\sigma_1) \}$, we have
$${\Cal L}(\sigma_1) = \sum_{i\in H}  
a_i  e_{l_i}   \rho_i\in\ker({\Frak s}_W)$$
and there is a $\Gamma^{(m)}$-homogeneous representation
$${\Cal L}(\sigma_1) = \sum_{\sigma\in\GM_W(F)} a_\sigma \sigma\ast
  \rho_\sigma,   w(\sigma)\circ\rho_\sigma = w(\sigma_1)$$
with $\rho_\sigma\in{\Cal B}, a_\sigma\in R\setminus\{0\}$.

Then
\begin{eqnarray*}
\sigma_2
&:=& 
\sigma_1 -  \sum_{\sigma\in\GM_W(F)}
a_\sigma\rho_\sigma\diamond \lift(\sigma)
\\&=& 
\sigma_1 -  \sum_{\sigma\in\GM_W(F)}
a_\sigma\rho_\sigma\diamond  \left(\sigma-\bar\sigma\right)
\\&=&
\sigma_1-{\Cal L}(\sigma_1)
+ \sum_{\sigma\in\GM_W(F)}a_\sigma\rho_\sigma\diamond
 \bar{\sigma}
\\&=& 
\sum_{i\not\in H} a_i e_{l_i} \rho_i
+\sum\limits_{\sigma\in\GM_W(F)} \sum_{i=1}^{\mu_\sigma}
\left(
a_\sigma  a_{i\sigma}\right)
e_{l_{i\sigma}}
\left(  \rho_{i\sigma}
\star \rho_\sigma\right)
\end{eqnarray*}
satisfies both $\sigma_2\in\ker({\Frak S}_W)$ and
$w(\sigma_2)<w(\sigma_1)$;
thus the claim follows by induction.
\end{proof}

\section{Weispfenning: Restricted Representation and Completion}\label{SWRRC}
Note that $R$ is effectively given as a quotient of a 
 free monoid ring ${\Cal R} := {\Bbb D}\langle {\Bcc v}\rangle$ over  ${\Bbb D}$ and the monoid $\langle {\Bcc v}\rangle$ of all words over the alphabet ${\Bcc v}$ modulo a bilateral ideal $I$,
 $R={\Cal R}/I$.

Wlog we will assume that $<$ orders the set ${\Bcc V}$ so that $X_1<X_2<\ldots$ and that its restriction to $\langle {\Bcc v}\rangle$ is a sequential term-ordering, {\em id est} the set $\{\omega\in\langle {\Bcc v}\rangle : \omega<\tau\}$ is finite for each $\tau\in\langle {\Bcc v}\rangle$.

Note  that, under these assumptions, (\ref{50cEqAlg}) implies the existence in ${\Cal A}$ of relations 
$$X_i\star d=\sum_{l=1}^{i} a_{li}(d) X_l+ a_{0i}(d), a_{li}(d)\in  {\Bbb D}\langle {\Bcc v}\rangle, \Forall X_i\in {\Bcc V}, d\in R\setminus\{0\}$$
and
$$\rho\star x_j=\sum_{\genfrac{}{}{0pt}{}{\upsilon\in{\Cal B}}{\upsilon\leq\rho}} a_{\rho j\upsilon}\upsilon, a_{\rho j\upsilon}\in  {\Bbb D}\langle {\Bcc v}\rangle, \Forall x_j\in {\Bcc v}, \rho\in{\Cal B}.$$
 
\begin{Lemma} \cite{W} Let 
$$F := \{g_1,\ldots,g_u\}\subset {\Cal A}^m, 
g_i = {\bf M}(g_i)-p_i =: c_i \tau_i {\bf e}_{\iota_i} - p_i;$$
set $\Omega := \max_<\{{\bf T}(g_i) : 1 \leq i\leq u\}$.

Let ${\sf M}$ be the bilateral module ${\sf M} := {\Bbb I}_2(F)$ 
and  ${\Bbb I}_W(F)$  the  restricted module 
\begin{eqnarray*}
{\Bbb I}_W(F) &:=& \Span_{R}(a f\star \rho   : a\in
R\setminus\{0\}, \rho\in {\Cal B}, f\in F)\\  &=& \Span_{R}(a\rho\diamond f  : a\in
R\setminus\{0\}, \rho\in {\Cal B}, f\in F).
\end{eqnarray*}
 
If every $g\star a_{\rho j\upsilon}, x_j\in {\Bcc v}, \upsilon,\rho\in{\Cal B},\upsilon\leq\rho<\Omega,$ has a restricted representation 
in terms of $F$ w.r.t. a sequential term-ordering  $<$, then
every $g\star r, g\in F, r\in{\Cal A}$, has a restricted representation in terms of $F$ w.r.t. $<$.
\end{Lemma}
\begin{proof} We can wlog assume $r=\prod_{l=1}^\nu x_{j_l}, x_{j_l}\in{\Bcc v}$ and prove the claim by induction on $\nu\in{\Bbb N}$.
 
 Thus we have a restricted representation in terms of $F$
 $$g\star\left(\prod_{l=1}^{\nu-1} x_{j_l}\right)=\sum_h d_h  g_{i_h}\star \rho_h, \tau_{i_h}\circ\rho_h\leq{\bf T}(g)\circ\prod_{l=1}^\nu x_{j_l},$$ whence we obtain
\begin{eqnarray*}
g\star \prod_{l=1}^\nu x_{j_l}&=&\left(g\star\prod_{i=1}^{\nu-1} x_{j_l}\right)\star x_{j_\nu}
\\ &=&
\left(\sum_h d_h  g_{i_h}\star \rho_h\right)\star x_{j_\nu}
\\ &=&
\sum_h d_h   g_{i_h}\star \left(\rho_h \star x_{j_\nu}\right)
\\ &=&
\sum_h d_h   g_{i_h}\star \left(\sum_{\genfrac{}{}{0pt}{}{\upsilon\in{\Cal B}}{\upsilon\leq\rho_h}} a_{\rho_h j_\nu\upsilon}\upsilon\right)
\\ &=&
\sum_h d_h \sum_{\genfrac{}{}{0pt}{}{\upsilon\in{\Cal B}}{\upsilon\leq\rho_h}}\left( g_{i_h}\star a_{\rho_h j_\nu\upsilon}\right)\upsilon
\end{eqnarray*}
and since $\upsilon\leq\rho_h<{\bf T}(f)\leq\Omega$  each element
$g_{i_h}\star a_{\rho_h j_\nu\upsilon}$ can be substituted with its restricted representation whose existence is granted by assumption.
\end{proof}

\begin{Lemma} \cite{W}  Under the same assumption, if,  for each 
$g\in F$, both each $X_i\star g, X_i\in{\Bcc V}$  and  each 
$g\star a_{\rho j\upsilon}, x_j\in {\Bcc v}, \upsilon,\rho\in{\Cal B},\upsilon\leq\rho<\Omega,$ have a restricted representation in terms of $F$ w.r.t. $<$, then ${\Bbb I}_W(F)={\sf M}$.
\end{Lemma} 

\begin{proof} It is sufficient to show that, for each $f\in {\Bbb I}_W(F)$, both each
$X_i\star f\in  {\Bbb I}_W(F), X_i\in{\Bcc V}$ and each $f\star x_j\in {\Bbb I}_W(F),x_j\in{\Bcc v}$.

By assumption $f=\sum_h d_h g_{i_h} \star \rho_h, d_h\in R\setminus\{0\}, \rho_h\in{\Cal B}\subset\langle {\Bcc Z}\rangle,1\leq i_h\leq u$, so that
\begin{eqnarray*}
X_i\star f &=& \sum_h \left(X_i\star d_h\right)g_{i_h} \star \rho_h
\\ &=&
 \sum_h\left(\sum_{l=1}^i a_{li}(d_h)X_l+ a_{0i}(d_h)\right)g_{i_h} \star \rho_h
\\ &=&
 \sum_h\sum_{l=1}^i a_{li}(d_h) \left(X_l\star g_{i_h}\right) \star \rho_h
+ \sum_h a_{0i}(d_h) g_{i_h} \star \rho_h
\end{eqnarray*}
and
\begin{eqnarray*}
f\star x_j&=&\sum_h d_h g_{i_h}\star\left(\rho_h \star x_j\right)
\\ &=&
\sum_h d_h   g_{i_h}\star \left(\sum_{\genfrac{}{}{0pt}{}{\upsilon\in{\Cal B}}{\upsilon\leq\rho_h}} a_{\rho_h j_\nu\upsilon}\upsilon\right)
\\ &=&
\sum_h d_h \sum_{\genfrac{}{}{0pt}{}{\upsilon\in{\Cal B}}{\upsilon\leq\rho_h}}\left( g_{i_h}\star a_{\rho_h j_\nu\upsilon}\right)\upsilon
\end{eqnarray*}
and, since $\upsilon\leq\rho_h<{\bf T}(f)\leq\Omega$  each element
$g_{i_h}\star a_{\rho_h j_\nu\upsilon}$ can be substituted with its restricted representation whose existence is granted by assumption.

The same holds for each $X_l\star g_{i_h}$ thus 
the claim follows.
\end{proof}

\begin{Corollary}\label{48cC1} \cite{W}  Let 
$$F := \{g_1,\ldots,g_u\}\subset {\Cal A}^m, 
g_i = {\bf M}(g_i)-p_i =: c_i \tau_i {\bf e}_{\iota_i} - p_i.$$
Let ${\sf M}$ be the bilateral module ${\sf M} := {\Bbb I}_2(F)$ 
and  ${\Bbb I}_W(F)$  the  restricted module 
\begin{eqnarray*}
{\Bbb I}_W(F) &:=& \Span_{R}(a f\star \rho   : a\in
R\setminus\{0\}, \rho\in {\Cal B}, f\in F)\\  &=& \Span_{R}(a\rho\diamond f  : a\in
R\setminus\{0\}, \rho\in {\Cal B}, f\in F).
\end{eqnarray*}

$F$ is the bilateral Gr\"obner basis of  ${\sf M}$ iff
\begin{enumerate}
\item denoting 
$\GM(F)$   any  restricted Gebauer--M\"oller set for $F$,  each $\sigma\in\GM(F)$ has a restricted quasi-Gr\"obner representation
in terms of $F$;
\item for each 
$g\in F$, both $X_i\star g, X_i\in{\Bcc V}$  and  each 
$$g\star a_{\rho j\upsilon}, x_j\in {\Bcc v}, \upsilon,\rho\in{\Cal B},\upsilon\leq\rho<\Omega,$$
 have a restricted representation in terms of $F$ w.r.t. $<$.
\end{enumerate}
\end{Corollary}

\section{Finiteness, Noetherianity, Termination}

Even if we restrict ourselves to a case in which both ${\Bcc v}$ and ${\Bcc V}$ are finite and that $<$ is a sequential term-ordering on
$\langle {\Bcc Z}\rangle$ so that the tests required by Corollary~\ref{48cC1} are finitely many,
unless we know and  explictly use noetherianity of ${\Cal A}$,
it is well-extablished that the best one can hope to be able of producing 
is a procedure which receiving as input a finite
set of elements $F := \{g_1,\ldots,g_u\} \subset {\Cal A}^m$
defining the module ${\Bbb I}(F)$
\begin{itemize}
\item in case ${\Bbb I}(F)$ has a finite (left, right, restricted, bilateral)  Gr\"obner basis,
halts returning such a finite  Gr\"obner basis;
\item otherwise, it produces an infinite sequence
of elements 
$$g_1,\ldots,g_u,g_{u+1},\ldots,g_i,\ldots$$ such that the infinite set 
$\{g_i: i \in \Bbb{N}\}$ is a  Gr\"obner basis of ${\Bbb I}(F)$.
\end{itemize}

A nice and efficient procedure to this aim has been proposed by Pritchard \cite{Pr2}, \cite[IV.47.7]{SPES}; with slight modification Pritchard's approach allows also to produce
\begin{itemize}
\item a procedure, which, given further an element $g\in{\Cal A}^m$, terminates if and only if $g\in{\Bbb I}(F)$
in which case it produces also a Gr\"obner representation of it;
\item a procedure, which, given an element $g\in{\Cal A}^m$ and any subset ${\Cal N}\subset{\bf N}({\Bbb I}(F))$, terminates if and only if $g\in{\Bbb I}(F)$ has a canonical representation $${\bf Rep}(g,{\Bbb I}(F))\subset\Span_{\Bbb D}({\Cal N})$$
in which case it produces such canonical representation, thus granting the impossibility of using non-commutative Gr\"obner bases as a cryptographical tool.
\end{itemize}

The procedures, assuming $<$ to be sequential, consists in fixing an enumerated set
$$\upsilon_1, \upsilon_2, \ldots, \upsilon_i, \upsilon_{i+1},\ldots$$
of the elements of $\langle {\Bcc Z}\rangle^{(m)}$ which satisfy
\begin{itemize}
\item $\upsilon_i<\upsilon_{i+1}$ for
each $i$,
\item for each $\upsilon\in\langle {\Bcc Z}\rangle^{(m)}$ there is a value $i : \upsilon < \upsilon_i$;
\end{itemize}
and denotes, for each $i\in{\Bbb N}$
$${\sf S}_i := \left\{\upsilon\leq\upsilon_i\right\}\subset{\sf S}^{(m)}$$

Then we set $G_0 := G, i:=1, S_0 :=\emptyset$ and iteratively we compute
\begin{itemize}
\item $B_i := \{\sigma\in\GM(G_{i-1}), w(\sigma)\leq\upsilon_i\}$,
\item $G_i := G_{i-1}\cup\left\{\NF({\Frak S}(\sigma),G_{i-1}) :\sigma\in B_i \right\}$.
\end{itemize}

\section{Restricted Gr\"obner basis}

In order to compute a restricted Gr\"obner basis we need to formulate Spear Theorem in the restricted setting.

It is more convenient to consider the ring
${\Cal Q}/I$ and the obvious projections 
$$\Phi : \left({\Cal Q}/I\right)^m \twoheadrightarrow {\Cal A}^m,
\ker(\Phi) = \left({\Cal I}/I\right)^m
={\Bbb I}_2\left(\pi(H)^{(m)}\right)$$ where $H=G\setminus\left(G_0\cup C\right)$
 and $\pi(H)^{(m)} := \{\pi(h){\bf e}_j,h\in H, 1\leq j \leq m\}.$

Then given a restricted module
${\sf M}:={\Bbb I}_W(F)\subset {\Cal A}^m$, where $F\subset{\bf Zach}_<({\Cal A})^{(m)})\subset{\Cal Q}^{(m)}$ and wlog $f=\Pi(f)$ for each $f\in F$, we consider 
the restricted module 
{\small\begin{eqnarray*}
{\sf M}' &:=& 
{\sf M}+\Span_{R}\left(\gamma\upsilon f\star \rho   : \gamma\in
{\Bbb D}\setminus\{0\},\upsilon\in\langle{\Bcc v}\rangle,\gamma\upsilon\notin{\bf M}(I), \rho\in\langle{\Bcc V}\rangle, f\in F\cup\pi(H)^{(m)}\right)\\&=& 
{\sf M}+\Span_{R}\left(\gamma\upsilon \rho\diamond f   : \gamma\in
{\Bbb D}\setminus\{0\},\upsilon\in\langle{\Bcc v}\rangle,\gamma\upsilon\notin{\bf M}(I), \rho\in\langle{\Bcc V}\rangle, f\in F\cup\pi(H)^{(m)}\right).
\end{eqnarray*}
} 

\begin{Lemma}[Spear] \cite{Sp},\cite[II.Proposition~24.7.3., IV.Theorem~50.6.3.(1)]{SPES}
With the present notation  if $F$ is a reduced restricted Gr\"obner basis of ${\sf M}'$, then 
$$\{g\in F : g = \Phi(g)\} =
\{\Phi(g) : g\in F, {\bf T}(g)\in{\Cal B}^{(m)}\} = F\cap{\bf Zach}_<({\Cal A})^m$$
is a reduced restricted Gr\"obner basis of ${\sf M}$.
\end{Lemma}

\begin{proof} 
Let $m\in{\sf M}$ and $m'\in{\sf M}'\cap {\bf Zach}_<({\Cal A})^m\subset{\Cal Q}^m$ be such that 
$\Phi(m') = m$, so that  
${\bf M}(m') = {\bf M}(m)\notin{\bf M}({\Cal I}^m)$, and $m'=\pi(m')$.

Then there are
$g_i\in F$, 
$\rho_i\in{\Bcc V}$, $\bar\upsilon_i\in\langle{\Bcc v}\rangle,
\gamma_i\in {\Bbb D}\setminus\{0\},\gamma_i\bar\upsilon_i\notin{\bf M}(I)$
such that, denoting ${\bf M}(g_i) = c_i\tau_i{\bf e}_{\iota_i}=c_i\upsilon_i\omega_i{\bf e}_{\iota_i}$, satisfy
\begin{itemize}
\renewcommand\labelitemi{\bf --}
\item ${\bf M}(m) := c\tau{\bf e}_\iota=c\upsilon\omega{\bf e}_\iota = 
\sum_i \gamma_i\bar\upsilon_i{\bf M}(g_i)\star \rho_i$,
\item $\tau = \bar\upsilon_i\cdot\tau_i\cdot\rho_i$,
\item $\omega=\omega_i\circ\rho_i$,
\item $\iota_i=\iota$,
\item $\Pi(g_i) = g_i$ and  ${\bf T}(g_i)\in{\Cal B}^{(m)}$.
\end{itemize}

Thus in particular we have
\begin{itemize}
\renewcommand\labelitemi{\bf --}
\item ${\bf T}(m) = {\bf T}(g_i)\circ\rho_i$ and
\item ${\bf M}(m) = \sum_i \gamma_i\bar\upsilon_i{\bf M}(g_i)\star\rho_i =   \sum_i\gamma_i\bar\upsilon_i\rho_i\diamond{\bf M}(g_i)$
\end{itemize}
as required.
\end{proof}

Let $F\subset {\Cal A}^m$ and express each $g\in F$ as 
$$g = {\bf M}(g)-p_g =: c_g \omega_g {\bf e}_{\iota_g} - p_g
=\left(\gamma_g\upsilon_g-\chi_g\right)\omega_g {\bf e}_{\iota_g} - p_g$$
with $$p_g\in{\Cal A}^m, c_g\in R, \omega_g\in\langle{\Bcc V}\rangle,
\chi_g\in R,\gamma_g\in{\Bbb D}, \upsilon_g\in\langle{\Bcc v}\rangle,$$ and
${\bf T}(p_g)<\tau_g, \gamma_g\upsilon_g\notin{\bf M}(I)$ and ${\bf T}(\chi_g)<\upsilon_g.$

Note that, analogously,  for each 
$h\in H:=G\setminus\{G_0\cup C\}\subset{\Cal Q}$, ${\bf M}(h)$ can be uniquely expressed as
$${\bf M}(h) = c_h\omega_h=(\gamma_h\upsilon_h+\chi_h)\omega_h$$
with $\gamma_h\in{\Bbb D}, \upsilon_h\in\langle{\Bcc v}\rangle,\omega_h\in\langle{\Bcc V}\rangle,c_h,\chi_h\in R,\gamma_h\upsilon_h\notin{\bf M}(I),{\bf T}(\chi_h)<\upsilon_h.$ 

In order to apply Spear's Theorem we  
adapt the notation of \cite[Corollary 14]{Bath} and consider
\begin{itemize}
\item the module  
$\left({\Cal Q}/I\right)^{\vert F\vert+m\vert H\vert}$ indexed by the set $F\cup\pi(H)^{(m)}$ and whose canonical basis is denoted 
$\{{\sf e}(f) : f\in F\cup\pi(H)^{(m)}\}$, and
\item $\hat{\Frak S}_2 : \left({\Cal Q}/I\right)^{\vert F\vert+m\vert H\vert} \to {\Cal A}^m : {\sf e}(h)\mapsto \Phi(h)$, for each $h\in F\cup G^{(m)}$.
\end{itemize}

Spear's Theorem having reduced the problem of computing restricted Gebauer-M\"oller sets to the classical problem of computing
Gebauer-M\"oller sets for elements in ${\Cal Q}$ with a restricted representation, we can on one side use the classical Buchberger Theory for Free Associative Algebras and, on the other side, take advantage of the restricted shape of the terms.

In particular, among two terms $\upsilon_1\omega_1,\upsilon_2\omega_2$
there is at most a single match and (by left and right cancellativity) either $\omega_1\mid_L\omega_2$ or $\omega_2\mid_L\omega_1$ and
either $\upsilon_1\mid_R\upsilon_2$ or $\upsilon_2\mid_R\upsilon_1$.

Thus, for 
$$g_1,g_2\in F, h\in H, {\bf M}(g_1) =\gamma_1\upsilon_1\omega_1 {\bf e}_{\iota_1},
{\bf M}(g_2) =\gamma_2\upsilon_2\omega_2 {\bf e}_{\iota_2}, {\bf M}(h) =\gamma_3\upsilon_3\omega_3, \omega_1\mid_L\omega_2$$ with $\gamma_i\in{\Bbb D},\upsilon_i\in\langle{\Bcc v}\rangle,\omega_i\in\langle{\Bcc V}\rangle,
\gamma_i\upsilon_i\notin{\bf M}(I)$ and 
$$\iota_1=\iota_2,\omega_1\mid_L \omega_2, \omega_1\rho= \omega_2, \rho\in{\Cal B}:$$
 
\begin{enumerate}
\renewcommand\theenumi{{\rm A.\arabic{enumi})}}
\item if $\omega_1\mid_L \omega_3, \omega_1\rho= \omega_3, \rho\in{\Cal B}$
and $\upsilon_3\mid_R \upsilon_1, \lambda\upsilon_3=\upsilon_1, \lambda\in\langle{\Bcc v}\rangle, 
\frac{\lcm(\gamma_1,\gamma_3)}{\gamma_3}\lambda\notin{\bf M}(I)$ 
 we set
$$B(g_1,h)=\frac{\lcm(\gamma_1,\gamma_3)}{\gamma_3}\lambda{\sf e}(h){\bf e}_{\iota_1}-\frac{\lcm(\gamma_1,\gamma_3)}{\gamma_1}\rho\diamond {\sf e}(g_1);$$
\item if $\omega_3\mid_L \omega_1, \omega_3\rho= \omega_1, \rho\in{\Cal B}$
and $\upsilon_1\mid_R \upsilon_3, \lambda\upsilon_1=\upsilon_3, \lambda\in\langle{\Bcc v}\rangle, 
\frac{\lcm(\gamma_1,\gamma_3)}{\gamma_1}\lambda\notin{\bf M}(I)$ 
 we set
$$B(g_1,h)=\frac{\lcm(\gamma_1,\gamma_3)}{\gamma_3}\rho\diamond{\sf e}(h){\bf e}_{\iota_1}-
\frac{\lcm(\gamma_1,\gamma_3)}{\gamma_1}\lambda{\sf e}(g_1);$$
\item if $\omega_1\mid_L \omega_3, \omega_1\rho= \omega_3, \rho\in{\Cal B}$
and $\upsilon_1\mid_R \upsilon_3, \lambda\upsilon_1=\upsilon_3, \lambda\in\langle{\Bcc v}\rangle, 
\frac{\lcm(\gamma_1,\gamma_3)}{\gamma_1}\lambda\notin{\bf M}(I)$ 
 we set
$$B(g_1,h)=\frac{\lcm(\gamma_1,\gamma_3)}{\gamma_3}{\sf e}(h){\bf e}_{\iota_1}-\frac{\lcm(\gamma_1,\gamma_3)}{\gamma_1}\lambda\rho\diamond {\sf e}(g_1);$$
\item if $\omega_3\mid_L \omega_1, \omega_3\rho= \omega_1, \rho\in{\Cal B}$
and $\upsilon_3\mid_R \upsilon_1, \lambda\upsilon_3=\upsilon_1, \lambda\in\langle{\Bcc v}\rangle, 
\frac{\lcm(\gamma_1,\gamma_3)}{\gamma_3}\lambda\notin{\bf M}(I)$ 
 we set
$$B(g_1,h)=\frac{\lcm(\gamma_1,\gamma_3)}{\gamma_3}\lambda\rho\diamond {\sf e}(h){\bf e}_{\iota_1}-\frac{\lcm(\gamma_1,\gamma_3)}{\gamma_1}{\sf e}(g_1);$$
\renewcommand\theenumi{{\rm B.\arabic{enumi})}}\setcounter{enumi}{0}
\item if $\omega_1\mid_L \omega_2, \omega_1\rho= \omega_2, \rho\in{\Cal B}$
and $\upsilon_2\mid_R \upsilon_1, \lambda\upsilon_2=\upsilon_1, \lambda\in\langle{\Bcc v}\rangle, 
\frac{\lcm(\gamma_1,\gamma_2)}{\gamma_2}\lambda\notin{\bf M}(I)$ 
 we set
$$B(g_1,g_2)=\frac{\lcm(\gamma_1,\gamma_2)}{\gamma_2}\lambda{\sf e}(g_2)-\frac{\lcm(\gamma_1,\gamma_2)}{\gamma_1}\rho\diamond {\sf e}(g_1);$$
\setcounter{enumi}{2}
\item if $\omega_1\mid_L \omega_2, \omega_1\rho= \omega_2, \rho\in{\Cal B}$
and $\upsilon_1\mid_R \upsilon_2, \lambda\upsilon_1=\upsilon_2, \lambda\in\langle{\Bcc v}\rangle, 
\frac{\lcm(\gamma_1,\gamma_2)}{\gamma_1}\lambda\notin{\bf M}(I)$ 
 we set
$$B(g_1,g_2)=\frac{\lcm(\gamma_1,\gamma_2)}{\gamma_2}{\sf e}(g_2)-\frac{\lcm(\gamma_1,\gamma_2)}{\gamma_1}\lambda\rho\diamond {\sf e}(g_1).$$
\end{enumerate}
 
\begin{Corollary} The set
$$\left\{B(f,g) : f,g\in F, \iota_f=\iota_g,\omega_f\mid_L \omega_g\right\}\cup\left\{B(f,h) : f\in F, h\in H\right\}$$
is a restricted Gebauer-M\"oller set.
\end{Corollary}

\section{Strong restricted Gr\"obner basis}

According Zacharias approach \cite{Z}, modules in ${\Cal A}$ have (left/right/bilate\-ral/restricted) strong Gr\"obner bases if and only if $R$ is a (left/right/bilateral/restricted) strong ring \cite{Zac}, {\em id est}  each  (left/right/bilateral/restricted) ideal $I\subset R$ has a strong basis.

Thus, under this assumption, from a restricted Gr\"obner basis $F\subset {\Cal A}^m$
of the restricted module 
${\Bbb I}_W(F)$, we can obtain a strong restricted Gr\"obner basis of ${\Bbb I}_W(F)$, as follows.

For each $g\in F$, let us denote  
\begin{itemize}
\renewcommand\labelitemi{\bf --}
\item $H_g:=\{h\in F\cup H^{(m)} : \omega_h\mid_L\omega_g\}$,
\item for each $h\in H_g, t_{hg}\in\langle{\Bcc V}\rangle : \omega_ht_{hg}=\omega_g$,
\item ${\sf J}_g :={\Bbb I}_L(\lc(h) : h\in H_g)\subset R$,
\item $\{d_j, j\in J\}, d_j=\sum_{h\in H_g} \gamma_{jh}\lc(h),$  a strong left basis of ${\sf J}_g$,
\item $S_g :=\{\sum_{h\in H_g} \Pi(\gamma_{jh}t_{hg})\diamond h,  j\in J\}$.
\end{itemize}

\begin{Corollary} $\cup_{g\in F}S_g$ is a strong restricted Gr\"obner representation in terms of $F$.
\end{Corollary}

\section*{Acknowledgements}
The senior author was partially supported by GNSAGA (INdAM, Italy).



\end{document}